\documentclass[12pt]{amsart}

\usepackage[utf8]{inputenc}
\usepackage{amstext,amsthm,amsmath}

\usepackage{latexsym}
\usepackage{amsfonts}
\usepackage{mathtools}
\usepackage{graphicx}
\usepackage[normalem]{ulem}
\usepackage{fullpage}
\usepackage{marvosym}
\usepackage[usenames,dvipsnames]{color}
\usepackage{verbatim}
\usepackage{cite}
\usepackage[hidelinks]{hyperref}
\usepackage[english]{babel}
\usepackage{tabularx}
\usepackage{tikz}
\usepackage{enumerate}

\usepackage{bigints}

\makeatletter
\renewcommand{\@captionfont}{\normalfont\small}
\makeatother

\newtheorem*{theorem*}{Theorem}
\newtheorem*{proposition*}{Proposition}
\newtheorem*{lemma*}{Lemma}
\newtheorem*{conjecture*}{Conjecture}
\newtheorem*{open problem*}{Open Problem}
\newtheorem*{problem*}{Problem}
\newtheorem{theorem}{Theorem}[section]
\newtheorem{lemma}[theorem]{Lemma}
\newtheorem{proposition}[theorem]{Proposition}
\newtheorem{corollary}[theorem]{Corollary}

\theoremstyle{definition}
\newtheorem{definition}[theorem]{Definition}
\newtheorem{remark}[theorem]{Remark}

\numberwithin{equation}{section}
\numberwithin{figure}{section}
\numberwithin{table}{section}

\def\diag{\operatorname{diag}}

\allowdisplaybreaks

\newcommand\vv{{{\bf{v}}}}
\newcommand\xx{{\boldsymbol{\chi}}}
\newcommand\ff{{\boldsymbol{f}}}
\newcommand\uu{{\boldsymbol{u}}}
\newcommand\gggg{{\boldsymbol{g}}}
\newcommand\del{{{\bf{\delta}}}}
\newcommand\RR{{\mathbb R}}
\newcommand\sump{{\sideset{}{'}\sum}}

\begin{document}
\title[Chebyshev approximations and reconstruction of signals]{Uniform Chebyshev approximations of functions satisfying a variation-type condition: reconstruction of time-varying signals on graphs}
\author{Davide Bianchi}
\address{School of Mathematics (Zhuhai), Sun Yat-Sen University, Zhuhai, Guangdong, P. R.
China}
\email{bianchid@mail.sysu.edu.cn}
\author{Sandra Saliani}
\address{Dipartimento di Ingegneria, Università degli Studi di Napoli Parthenope, Naples, Italy}
\email{sandra.saliani@uniparthenope.it}
\author{Dimitrios Vavitsas}
\address{School of Mathematics (Zhuhai), Sun Yat-Sen University, Zhuhai, Guangdong, P. R.
China}
\email{vavitsas@mail.sysu.edu.cn}

\date{}
\keywords{Chebyshev polynomials, graph theory, wavelets, spectral graph theory, denoising}
\subjclass[2020]{41A10; 41A63; 42C15; 94A12; 05C50; 65J22}

\begin{abstract}
We prove uniform approximation theorems for Chebyshev expansions of continuous functions of two variables. Under a variation-type condition on the square $[-1,1]^2$, a continuous function admits a uniformly convergent Chebyshev expansion in the first variable whose coefficients are continuous functions of the second variable. These results are applied to the spectral graph wavelet transform of time-varying signals on finite weighted graphs: the scaling and wavelet kernels are approximated by expansions with time-varying coefficients whose degrees do not depend on time, the composition of the approximate transform with its adjoint admits the same explicit coefficient formulas as in the time-independent case, and reconstruction by the pseudoinverse is stable, with explicit bounds in terms of the uniform kernel errors. Numerical experiments on a sensor network confirm the convergence and stability estimates. In a denoising problem, soft thresholding of the graph wavelet coefficients with a time-varying transform parameter improves over its time-independent, fixed-parameter counterpart.
\end{abstract}

\maketitle
\section{Introduction}

This paper explores time-varying signals defined at the vertices of a graph, together with their analysis and reconstruction using the spectral graph wavelet transform \cite{Hammond}. The aim is to find a good approximation of the two-variable scaling and wavelet functions using Chebyshev polynomials. Indeed, in a time-independent setting, the scaling and wavelet functions on the graph are obtained as restrictions of one-variable functions defined on a real interval, and they are called kernels, even though they may initially be defined only on the discrete spectrum of the graph Laplacian; for example, restrictions to the spectrum of either splines or the Meyer kernel are usually considered.

Furthermore, for very large graphs, when it is computationally expensive to numerically calculate the entire spectrum of the Laplacian, it is necessary to approximate the kernels, and this procedure is efficiently achieved by Chebyshev polynomials, see \cite{Hammond, Saliani1,BS2026}. In the case of time-varying signals, one can think of the kernels as being defined on a rectangle in $\RR^2$, and it is therefore necessary to study how the whole approximation procedure can be extended to two variables.

In particular, our objective is to determine whether a two-variable function can be uniformly approximated by suitable finite Chebyshev expansions and to apply these results within the existing reconstruction framework for graph signals.

To this end, we first introduce the necessary background on weighted graphs, time-varying signals,
and spectral graph wavelet transforms. By a weighted graph, we refer to a finite, connected,
undirected graph with no loops or parallel edges and a positive real-valued weight function. In recent
decades, wavelets on graphs via spectral graph theory have attracted considerable attention and
provide a fruitful counterpart to classical wavelets on the real line in the context of Fourier
analysis, see \cite{Chui} for an overview from Fourier analysis to wavelet analysis. Eigenfunctions of the (combinatorial)
graph Laplacian play a role in graph theory similar to that played by the eigenfunctions of the
one-dimensional Laplace operator in Fourier analysis. Graph spectral theory has become of
considerable importance since in many real-world problems signals on graphs can be analyzed through
transforms defined on the spectrum, see \cite{Saliani1,BCPS2026,Hammond}, and the
references therein.

Next, we introduce the necessary background on Chebyshev polynomial approximation theory and
present the results that will subsequently be used in the reconstruction process. Chebyshev polynomials are of central importance in modern developments including orthogonal polynomials, polynomial approximation, numerical integration and PDEs, see \cite{Mason2} for more on the topic.

In \cite{Hammond}, the authors presented a fast Chebyshev polynomial approximation algorithm for weighted graphs, which avoids the need to diagonalize the graph Laplacian $\mathcal{L}$, thereby enabling signal reconstruction up to computational errors. In particular, given a finite weighted graph with corresponding Laplacian $\mathcal{L}$, given a signal $\ff\in \RR^N$, given polynomial approximants $p_j$, where $p_0$ approximates the scaling function $h(x)$ and $p_j$, $j\geq 1$, approximates the dilated wavelet function $g(s_jx)$, and given the corresponding scaling/wavelet coefficients $\tilde{W}_{\ff}(s_j,n)=(p_j(\mathcal{L})\ff)_n$, the following computation is needed:
\begin{equation*}
\tilde{W}^*\tilde{W}\ff=\Big(\sum_{j}(p_j(\mathcal{L}))^2\Big)\ff,
\end{equation*}
where $\tilde{W}^*$ denotes the adjoint. Assuming that each approximant $p_j$ is a finite Chebyshev expansion, the equation above leads to the following computation:
\begin{equation}\label{prin}
 \tilde{W}^*\tilde{W}\ff=\frac{1}{2}d_0\ff+\sum_{k=1}^{M^*}d_k\overline{T}_k(\mathcal{L})\ff,
\end{equation}
where the coefficients $d_k$ are computed from the coefficients of the expansions of the $p_j$ in the shifted Chebyshev basis, see Sections 6--7 of \cite{Hammond}.

Our goal is to consider time-varying signals $
\ff(\tau)\in \RR^N$ on graphs, namely signals that depend on a variable $\tau \in [0,\tau^*]$, and to employ the analogous scheme for computing the composition $\tilde{W}^*[\tau]\tilde{W}[\tau]\ff(\tau)$. Let us point out that in this framework the overall Graph Wavelet Transform $W$ depends on scaling and wavelet functions of the form $h(x,\tau)$, $g(s_jx,\tau)$, which raises the question of whether they can be approximated by suitable two-variable Chebyshev expansions.

In order to achieve this, we present Chebyshev approximation theorems. These results may be of
independent interest in approximation theory, beyond their application in the present work. More
precisely, we shall make use of the following theorem in the approximation of the scaling/wavelet
functions (the novel notion of variation-type condition is given in Definition~\ref{def var}):
if $f\in C([-1,1]^2)$ satisfies a variation-type condition, then $f$ has a uniformly convergent series expansion on $[-1,1]^2$ of the form
\begin{equation}\label{AWA}
f(x,y)=\sump_{k=0}^{\infty}a_k(y)T_k(x),
\end{equation}
where
\[a_k(y):=\frac{2}{\pi}\int_{-1}^{1}\frac{f(t,y)T_k(t)}{\sqrt{1-t^2}}dt.\]
Here and throughout the paper, the prime in $\sump$ indicates that the term corresponding to $k=0$ is halved.

The scaling/wavelet functions are usually chosen to be sufficiently regular, thus satisfying the ``variation-type'' condition of the statement. For instance, any $C^1$ function satisfies this condition. We shall show that if $g\in C([-1,1]^2)$ is $(m+1)$-times continuously differentiable with respect to $x$, then we have the error bound
\begin{equation*}
\sup_{x,y\in [-1,1]}\{|g(x,y)-\mathcal{T}_{K}(g)(x,y)|\}\leq \frac{C}{K^m},
\end{equation*}
for all $K\in \mathbb N$, where the constant $C>0$ does not depend on $K$, and $\mathcal{T}_{K}(g)$ is the $K$-th partial sum of \eqref{AWA} for $g$.
More information on Chebyshev polynomials and variation-type conditions follows in Section~\ref{Back UCA}.

Our study, concerning the approximation of functions on the square $[-1,1]^2$ either pointwise or uniformly by Chebyshev expansions, is also motivated by results in approximation theory.
Moreover, such approximation theorems may also be useful in regression analysis, see \cite{Breiman,Friedman,Paciorek}, and in image denoising and PDE-based image processing, see \cite{Buades,RudinL1,RudinL2}.

The idea is to obtain an analogue of the computation \eqref{prin}, where $M^*$ will not depend on the variable $\tau$, namely,
\[ \tilde{W}^*[\tau]\tilde{W}[\tau]\ff(\tau)=\sump_{k=0}^{M^*}\overline d_k(\tau)\overline T_k(\mathcal{L})\ff(\tau),\]
for all $\tau\in [0,\tau^*]$. The time-varying coefficients $\overline d_k(\tau)$ fully characterize the approximants at each time instant and may be computed directly once the scaling/wavelet functions are given.

This differs from the pointwise approximation and from the discrete setting, in which the degree $M^*$ depends on the number of time points considered.

In the present paper, we aim to develop a more theoretical framework for signals and scaling/wavelet functions that depend on a continuous time variable $\tau$, and to adapt the signal reconstruction scheme of D.K. Hammond, P. Vandergheynst and R. Gribonval to this setting.

The structure of the paper is as follows. First, in Section~\ref{Back SGT}, we present the necessary background on wavelets on graphs via spectral graph theory, suitably adapted to the time-varying setting. Next, in Section~\ref{Back UCA}, we briefly introduce Chebyshev polynomials in one and two variables, and then we give the novel notion of variation-type condition, which plays a key role in the proof of the uniform Chebyshev approximation theorem used in this work. We further deal with uniform error bounds for such approximation expansions. Next, in Section~\ref{Back Appl}, we apply the Chebyshev approximation theorem and the signal reconstruction scheme to the time-varying setting. In Section~\ref{sec:numerical}, we verify the approximation and stability estimates and apply the complete transform to soft-threshold denoising of a time-varying signal on a sensor network. Finally, brief conclusions are collected in Section~\ref{sec:conclusions}.

\section{Background on Spectral Graph Theory}\label{Back SGT}

In this section, we collect the background material needed in the sequel: weighted graphs and their Laplacians, the Graph Fourier Transform, and the overall Graph Wavelet Transform of time-varying signals, together with the basic properties that enter the signal reconstruction scheme.

\subsection{Weighted Graphs and Laplacian}
Throughout this work, $G=(I,E,w)$ denotes a finite, connected, undirected weighted graph. Its vertex set is $I=\{1,\dots,N\}$, where $N\geq 2$, while
\[
E\subseteq \bigl\{\{i,j\}:i,j\in I,\ i\neq j\bigr\}
\]
is the set of edges. Thus, $G$ has neither loops nor parallel edges. For $i,j\in I$, we write $i\sim j$ whenever $\{i,j\}\in E$. Connectedness means that, for every pair of vertices $i,j\in I$, there exists a finite sequence
\[
i=i_0\sim i_1\sim\cdots\sim i_k=j.
\]

The weight function
\[
w:E\longrightarrow (0,\infty)
\]
assigns a positive weight to each edge. The adjacency matrix associated with $G$ is the symmetric matrix $\mathsf{A}=(a_{ij})_{i,j=1}^N$ defined by
\[
a_{ij}:=
\begin{cases}
w(\{i,j\}), & \text{if } i\sim j,\\
0, & \text{otherwise}.
\end{cases}
\]
In particular, since $G$ has no loops, $a_{ii}=0$ for every $i\in I$. We refer to \cite{Bapat,Budy} for further background on graph theory.

Let $G$ be such a graph and let $D$ be the $N\times N$ \emph{degree diagonal matrix} with entries $d_{ii}:=\sum_{j=1}^{N}a_{ij}>0$. The \emph{unnormalized Laplacian} is defined to be the matrix
\[ \mathcal{L}=D-\mathsf{A}.\]
In particular, $\mathcal{L}$ is a real symmetric, diagonally dominant matrix of the form
\[\mathcal{L}=\begin{pmatrix}
 d_{11} &-a_{12}&\cdots &-a_{1N}\\
 -a_{21}&d_{22}&\cdots &-a_{2N}\\
\vdots& \vdots&\cdots&\vdots\\
-a_{N1}&-a_{N2}&\cdots & d_{NN}
\end{pmatrix}.\]

Let us recall some information about the eigenvalues. The Laplacian $\mathcal{L}$ has non-negative, real eigenvalues which, since the graph is connected, satisfy \[0=\lambda_0<\lambda_1\leq \cdots \leq \lambda_{N-1}.\]
One may deduce the non-negativity of the eigenvalues either from the explicit form of the Laplacian or, more generally, from Ger\v{s}gorin's theorem. The Laplacian always has $\lambda=0$ as an eigenvalue, since $(1,\dots,1)^\top\in \RR^N$ is a corresponding eigenvector. Lastly, it is well known that $\lambda_1>0$ if and only if the graph is connected.
By the spectral theorem for real symmetric matrices, $\mathcal{L}$ is orthogonally similar to $\Lambda=\diag(\lambda_0,\dots,\lambda_{N-1})$, that is, there exists an orthogonal matrix $U$ such that
\[ \mathcal{L}=U\Lambda U^\top.\]

The orthogonal matrix $U$ may be chosen as follows:
\[U=[\xx_0,\dots,\xx_{N-1}]:=\begin{pmatrix}
 \chi_0(1)&\chi_1(1)&\cdots &\chi_{N-1}(1)\\
 \chi_0(2)&\chi_1(2)&\cdots &\chi_{N-1}(2)\\
\vdots& \vdots&\cdots&\vdots\\
\chi_{0}(N)&\chi_{1}(N)&\cdots & \chi_{N-1}(N)
\end{pmatrix},\]
where each $\xx_\ell^\top=(\chi_\ell(1),\dots,\chi_\ell(N))$, $\xx_\ell\in \RR^N$, $\ell=0,\dots,N-1$, is the eigenvector corresponding to the eigenvalue $\lambda_\ell$. Note that $\xx_0^\top=(1/\sqrt{N},\dots,1/\sqrt{N})$. See \cite{Bapat,Horn} for background on matrix theory.

The spectral constructions and the approximation and stability arguments below also apply to the symmetric normalized Laplacian $\mathcal L_{\mathrm{sym}}:=D^{-1/2}\mathcal LD^{-1/2}=I_N-D^{-1/2}\mathsf A D^{-1/2}$; see \cite[Section~3.1]{Hammond}. This matrix is real symmetric and positive semidefinite, with spectrum contained in $[0,2]$, so one may take $\lambda^*=2$ in Section~\ref{Back Appl}. For a connected graph, its normalized zero eigenvector is $\xx_0=D^{1/2}\mathbf1/(\mathbf1^\top D\mathbf1)^{1/2}$, and wavelets with $g(0,\tau)=0$ are orthogonal to this degree-weighted vector.

\subsection{Graph Fourier Transform of a time-varying signal}
Consider a real-valued function $f:I\times [0,+\infty) \rightarrow \RR$, defined on the Cartesian product of the vertex set $I$ of a graph $G=(I,E,w)$ and the time interval $[0,+\infty)$. For each $\tau\in [0,+\infty)$, the map $f(\cdot, \tau)$ can be viewed as a vector $\ff( \tau)\in \RR^N$, that is, $\ff(\tau)^\top:=(f(1, \tau),\dots,f(N, \tau))$. Henceforth, we use this identification and the vector-valued function $\ff( \tau)$ is called a \emph{time-varying signal}; no regularity with respect to $\tau$ is assumed on the signal.

The definition of the Graph Fourier Transform of a time-varying signal mirrors the known definition for a graph signal,
see \cite{Saliani1,Grassi,Hammond,Sheikh,Shuman}.

One applies the Graph Fourier Transform to the signal $\ff(\tau)$ for each fixed time $\tau\in [0,+\infty)$, namely, we define the \emph{Graph Fourier Transform} of a time-varying signal as:
\[ \hat{\ff}(\tau):=U^\top\ff(\tau), \quad \tau\in [0,+\infty),\]
where $U=[\xx_0, \xx_1,\dots,\xx_{N-1}]$ is the orthogonal matrix of the Laplacian's decomposition. In particular,
\[
\hat{f}(\ell, \tau)=\langle \ff(\tau),\xx_\ell\rangle =\sum_{i=1}^{N}f(i,\tau)\chi_{\ell}(i), \quad \ell=0,\dots,N-1.
\]

The \emph{inverse Graph Fourier Transform} is then defined on spectral data ${\bf{q}}\in \RR^N$ by
\[ {\bf{q}}^\vee:=U{\bf{q}},\]
so that $\ff(\tau)=(\hat{\ff}(\tau))^\vee$ and
\[
f(i,\tau)=\sum_{\ell=0}^{N-1}\hat f(\ell, \tau)\chi_\ell(i), \quad i=1,\dots,N.
\]

It is easy to deduce the Parseval relation:
\[
    \langle \ff(\tau), \gggg (\tau)\rangle=\langle \hat \ff(\tau), \hat \gggg(\tau) \rangle,
\]
for all signals $\ff(\tau),\gggg(\tau)$ and all $\tau\in [0,+\infty)$.

\subsection{Overall Graph Wavelet Transform of a time-varying signal}

We define the \emph{Graph Wavelet Operator} at a scale $s>0$ and time $\tau\in [0,+\infty)$ with respect to a kernel function $g:[0,+\infty)\times[0,+\infty)\rightarrow \RR$ as:
\begin{equation}\label{WavOp}
    T_g[s,\tau]:=g(s\mathcal{L},\tau)=U\diag(g(s\lambda_0,\tau),\dots,g(s\lambda_{N-1},\tau))U^\top.
\end{equation}
When $\tau$ is  fixed, the above reduces to the analogous operator defined in \cite{Hammond}; we would like to stress that the kernel $g$ varies in time, provided $g$ is not constant in the second variable.

In particular, given a signal $\ff(\tau)$, the operator $T_g[s,\tau]$ acts as follows:
\[
(T_g[s,\tau]\ff(\tau) )(i)=\sum_{\ell=0}^{N-1} g(s\lambda_\ell,\tau)\hat f(\ell,\tau)\chi_\ell(i), \quad i=1,\dots,N.
\]

\begin{remark}
By \eqref{WavOp} it follows that the operators $T_g[s,\tau]$ are self-adjoint.
\end{remark}

When $g(0,\tau)=0$ for all $\tau$, we can define the translations of an $s$-scaled wavelet as follows.

Let $\delta_j:I\times [0,+\infty) \rightarrow \RR$ be such that $\delta_j(i,\tau):=1$ if $j=i$ and $\delta_j(i,\tau):=0$ otherwise.
Write $\boldsymbol{\del}_j(\tau)=\boldsymbol{\del}_j$ since it does not depend on $\tau$.
Its Fourier Transform satisfies
\[ \hat \delta_j(\ell ,\tau)=\sum_{i=1}^N\delta_j(i,\tau)\chi_\ell(i)=\chi_\ell(j).\]

Then the $j$-translation of an $s$-scaled graph wavelet is defined as
\[\psi_{g,j}[s,\tau]:=T_g[s,\tau]\boldsymbol{\del}_j,\]
so that
\begin{align*}
(\psi_{g,j}[s,\tau])(i)&=\sum_{\ell=0}^{N-1} g(s\lambda_\ell,\tau)\hat \delta_j(\ell,\tau)\chi_\ell(i)\\
&=\sum_{\ell=0}^{N-1} g(s\lambda_\ell,\tau)\chi_\ell(j)\chi_\ell(i).
\end{align*}

The same argument applied to $h:[0,+\infty)\times[0,+\infty)\rightarrow \RR$, with $h(0,\tau)>0$ and $h(x,\tau)\rightarrow 0$ as $x\rightarrow +\infty$, gives the \emph{Graph Scaling Operator}
\[
    T_h[1,\tau]:=h(\mathcal{L},\tau)=U\diag(h(\lambda_0,\tau),\dots,h(\lambda_{N-1},\tau))U^\top,
\]
and the $j$-translation of the graph scaling function is
\[\varphi_{h,j}[1,\tau]:=T_h[1,\tau]\,\boldsymbol{\del}_j.\]

For a fixed scale $s$, the graph wavelet coefficients of a time-varying signal are provided by:
\begin{align*}
\langle \psi_{g,j}[s,\tau],\ff(\tau)\rangle&=\sum_{i=1}^{N}\sum_{\ell=0}^{N-1} g(s\lambda_\ell,\tau)\chi_\ell(j)\chi_\ell(i)f(i,\tau)\\
&=\sum_{\ell=0}^{N-1} g(s\lambda_\ell,\tau)\hat f(\ell,\tau)\chi_\ell(j)\\
&=(T_g[s,\tau]\ff(\tau) )(j).
\end{align*}

Finally, given scales $s_1,\dots,s_r>0$, we define the \emph{overall Graph Wavelet Transform} of time-varying signals at time $\tau\in [0,+\infty)$ as the map $W[\tau]:\RR^N\rightarrow \RR^{N(r+1)}$ given by
\begin{equation}\label{GWT}
W[\tau]\ff(\tau)=\Big((T_h[1,\tau]\ff(\tau))^\top,(T_g[s_1,\tau]\ff(\tau))^\top,\dots,
(T_{g}[s_r,\tau]\ff(\tau))^\top\Big)^\top.
\end{equation}

See \cite{Saliani1,Grassi,Hammond,Sheikh,Shuman} for analogous definitions for time-independent signals.

The following remarks are taken from \cite{Hammond}; we include the proofs to illustrate the crucial role of the composition in the signal reconstruction scheme.

\begin{remark}
The overall transformation has adjoint $W^*[\tau]: \RR^{N(r+1)}\rightarrow \RR^N$ given by
\[ W^*[\tau]{\bf{v}}(\tau)=T_h[1,\tau]\vv_0(\tau)+\sum_{j=1}^{r}T_g[s_j,\tau]\vv_j(\tau),\]
where $\vv(\tau)=(\vv_0(\tau)^\top,\dots, \vv_r(\tau)^\top)^\top\in \RR^{N(r+1)}$.
\begin{proof}Indeed, this is a consequence of the self-adjointness of the operators $T_h[1,\tau]$ and $T_g[s_j,\tau]$:
\begin{align*}
\langle W[\tau]\ff(\tau),\vv(\tau)\rangle&=\langle T_h[1,\tau]\ff(\tau),\vv_0(\tau)\rangle+\sum_{j=1}^r\langle T_g[s_j,\tau]\ff(\tau),\vv_j(\tau)\rangle\\
&=\langle \ff(\tau),T_h[1,\tau]\vv_0(\tau)\rangle+\sum_{j=1}^r\langle\ff(\tau), T_g[s_j,\tau]\vv_j(\tau)\rangle.
\end{align*}
By the uniqueness of the adjoint, we get the desired result.
\end{proof}
\end{remark}

\begin{remark}\label{compos overall}
The composition $W^*[\tau]W[\tau]:\RR^N\rightarrow \RR^N$ satisfies
\[ W^*[\tau]W[\tau]\ff(\tau)=\left(h(\mathcal{L},\tau )^2+\sum_{j=1}^rg(s_j\mathcal{L},\tau)^2\right)\ff(\tau).\]
\begin{proof}Indeed, we have
\[h(\mathcal{L},\tau)\ff(\tau)=\Big(\sum_{\ell=0}^{N-1}h(\lambda_\ell,\tau)\hat f(\ell,\tau)\chi_\ell(1),\dots,\sum_{\ell=0}^{N-1}h(\lambda_\ell,\tau)\hat f(\ell,\tau)\chi_\ell(N)\Big)^\top,\]
and hence,
\begin{align*}
(T_h[1,\tau]h(\mathcal{L},\tau)\ff(\tau))(i)&=\sum_{\ell'=0}^{N-1}\sum_{\ell=0}^{N-1}\sum_{j=1}^{N}h(\lambda_{\ell'},\tau)h(\lambda_{\ell},\tau)\hat f(\ell,\tau) \chi_\ell(j)\chi_{\ell'}(j)\chi_{\ell'}(i)\\
&=\sum_{\ell=0}^{N-1}h(\lambda_{\ell},\tau)^2\hat f(\ell,\tau) \chi_\ell(i).
\end{align*}
Whence,
\[ T_h[1,\tau]h(\mathcal{L},\tau)\ff(\tau)=h(\mathcal{L},\tau)^2\ff(\tau).\]
Similarly,
\[ T_g[s_j,\tau]g(s_j\mathcal{L},\tau)\ff(\tau)=g(s_j\mathcal{L},\tau)^2\ff(\tau).\]
Furthermore,
\begin{align*}
 W^*[\tau]W[\tau]\ff(\tau)&=W^*[\tau]\Big((h(\mathcal{L},\tau)\ff(\tau))^\top,(g(s_1\mathcal{L},\tau)\ff(\tau))^\top,\dots, (g(s_r\mathcal{L},\tau)\ff(\tau))^\top\Big)^\top\\
 &=T_h[1,\tau]h(\mathcal{L},\tau)\ff(\tau)+\sum_{j=1}^rT_g[s_j,\tau]g(s_j\mathcal{L},\tau)\ff(\tau)\\
 &=\left(h(\mathcal{L},\tau)^2+\sum_{j=1}^rg(s_j\mathcal{L},\tau)^2\right)\ff(\tau).
\end{align*}
\end{proof}
\end{remark}

\begin{definition}
A sequence $\{f_k\}_{k=1}^{\infty}$ of elements in a Hilbert space $\mathcal{H}$ is a frame for $\mathcal{H}$ if there exist constants $A,B>0$ such that
\[A\lVert f\rVert^2\leq \sum_{k=1}^{\infty}|\langle f,f_k\rangle|^2\leq B\lVert f\rVert^2, \quad \forall f\in \mathcal{H}.\]
The numbers $A,B$ are called frame bounds.
\end{definition}
See \cite{frame} for background on frames.

Following the arguments in \cite[Theorem 5.6]{Hammond}, consider finitely many scales $s_k$, $k=1,\dots,r$, a fixed $\tau^*>0$, and set
\[G_\tau(\lambda):=h(\lambda,\tau)^2+\sum_{k=1}^rg(s_k\lambda,\tau)^2.\]
Assume that $h$ and $g$ are jointly continuous and that $h(\lambda,\tau)>0$ for all $(\lambda,\tau)\in [0,\lambda_{N-1}]\times [0,\tau^*]$. Then, for each fixed $\tau\in [0,\tau^*]$, the family $\{\varphi_{h,j}[1,\tau]\}_{j=1}^{N}\cup\{\psi_{g,j}[s_k,\tau]\}_{j,k=1}^{N,r}$ forms a frame for $\RR^N$ with bounds
\[A(\tau):=\min_{\lambda\in [0,\lambda_{N-1}]}G_\tau(\lambda)>0,\]
\[B(\tau):=\max_{\lambda\in [0,\lambda_{N-1}]}G_\tau(\lambda).\]
The optimal frame bounds are given by $\min_{0\leq \ell\leq N-1}G_\tau(\lambda_\ell)$ and $\max_{0\leq \ell\leq N-1}G_\tau(\lambda_\ell)$; taking the minimum and the maximum over the whole interval $[0,\lambda_{N-1}]$ yields valid, though in general non-optimal, bounds thanks to the positivity of $A(\tau)$. Furthermore, since $G_\tau(\lambda)$ is continuous on the compact set $[0,\lambda_{N-1}]\times [0,\tau^*]$, we obtain frame bounds independent of $\tau$:
\[0<\min_{(\lambda,\tau)\in [0,\lambda_{N-1}]\times [0,\tau^*]}G_\tau(\lambda)\leq A(\tau)\leq B(\tau)\leq \max_{(\lambda,\tau)\in [0,\lambda_{N-1}]\times [0,\tau^*]}G_\tau(\lambda)<\infty.\]

As we have already mentioned, one of the main goals is to efficiently reconstruct signals from the values of the chosen transform. Any formula that expresses every signal with respect to such a transform is commonly called an ``inverse formula'' (see \cite[Lemma 5.1]{Hammond} for such an example).

The formulation of an inverse formula usually requires the exact knowledge of all Laplacian eigenvalues, and for arbitrarily large graphs, this task may be extremely challenging, if not impossible.

In \cite{Hammond}, this problem is overcome for large graphs by utilizing Chebyshev polynomial approximation of both the scaling and the wavelet kernels $h,g$, and hence of the pseudoinverse of the Spectral Graph Wavelet Transform. See \cite{frame} for an introduction to pseudoinverse operators.

The same arguments may be applied to this setting, the pseudoinverse being
\[(W^*[\tau]W[\tau])^{-1}W^*[\tau].\]
It is then possible to recover the signal up to computational errors.

The polynomial approximation may be taken over a compact set containing the spectrum of $\mathcal{L}$.

\begin{remark}\label{approx}
Let $\lambda_{\text{max}}\geq\lambda_{N-1}$ and $\tau^*>0$. Fix $s>0$ and let $p(x,\tau)$ be a polynomial approximant of the function $g(sx,\tau)$ with an error
$S:=\sup_{x\in [0,\lambda_{\text{max}}],\tau\in [0,\tau^*]}\{|g(sx,\tau)-p(x,\tau)|\}$. Then $T_p[s,\tau]\ff(\tau):=p(\mathcal{L},\tau)\ff(\tau)$ satisfies
\[|(T_g[s,\tau]\ff(\tau))(i)- (T_p[s,\tau]\ff(\tau))(i)|\leq S\lVert \ff(\tau)\rVert, \quad \tau\in [0,\tau^*].\]
\begin{proof}
\begin{align*}
|(T_g[s,\tau]\ff(\tau))(i)- (T_p[s,\tau]\ff(\tau))(i)|&=|\sum_{\ell=0}^{N-1}(g(s\lambda_\ell,\tau )-p(\lambda_\ell,\tau))\hat f(\ell,\tau)\chi_\ell(i) |\\
&\leq S\lVert \ff(\tau)\rVert,
\end{align*}
where the last inequality follows from the Cauchy--Schwarz inequality, the orthonormality of the rows of $U$, and the identity $\lVert \hat \ff(\tau)\rVert=\lVert \ff(\tau)\rVert$.
\end{proof}
\end{remark}

\section{Uniform Chebyshev approximations}\label{Back UCA}

We are ready to proceed to the first objective of this work concerning uniform Chebyshev approximations. In this section, we present the necessary background on Chebyshev polynomials and then state and prove the main results.

In order to relate this section to the application to signal reconstruction, consider the following key question arising in \cite{Hammond}. To avoid confusion, assume that the signal $\ff$ does not depend on $\tau$. Let $p_0,p_j$ be polynomial approximants to a scaling function $h(x)$ and to the wavelet functions $g(s_j\,\cdot)$, respectively, and consider the approximant composition
\begin{equation*}
\tilde{W}^*\tilde{W}\ff:=\Big(\sum_{j}(p_j(\mathcal{L}))^2\Big)\ff.
\end{equation*}
Let $p_j=\sum_{k}a^j_ke_k$, where $\{e_k\}$ forms an orthogonal system in a suitable function space. Is it possible to write $p_j^2=\sum_kd_k^je_k$ and deduce the expansion
\begin{equation*}
\tilde{W}^*\tilde{W}\ff:=\sum_kd_ke_k(\mathcal{L})\ff,
\end{equation*}
where the coefficients $d_k$ are computed directly from the $a_k^j$?

Chebyshev polynomial theory provides an affirmative answer to this question, and it serves as a suitable and efficient polynomial approximation tool for large graphs.

Our goal is to seek suitable approximant polynomials analogous to the ones obtained in \cite{Hammond} and to apply the theory to obtain the approximant composition
\begin{equation*}
\tilde{W}^*[\tau]\tilde{W}[\tau]\ff(\tau):=\sum_kd_k(\tau)e_k(\mathcal{L})\ff(\tau).
\end{equation*}

\subsection{Chebyshev polynomials in \texorpdfstring{$\RR$}{R}}
The Chebyshev polynomials of the first kind $T_k, k=0,1,2,\dots$, are polynomials
in the variable $x$ of degree $k$, and they are defined by the relation
\[T_k(x) = \cos(k\theta), \quad  x=\cos\theta.\]

The family $\{T_k\}_{k=0}^{\infty}$ forms an orthogonal system on $L^2([-1,1],\frac{dt}{\sqrt{1-t^2}})$, the space of all square-integrable functions with respect to the measure $\frac{dt}{\sqrt{1-t^2}}$.

Denote by $C(X)$ the space of continuous functions on a set $X$. Moreover, recall the notation
\[\sump_{k=0}^{K}c_k:=\frac{1}{2}c_0+c_1+c_2+\cdots+c_K,\]
which is used analogously for infinite sums. Any function $f\in C([-1,1])$ which is either of bounded variation or satisfies a Dini–Lipschitz condition on $[-1,1]$ has a uniform Chebyshev series expansion, namely
\[f(x)=\sump_{k=0}^{\infty}a_kT_k(x),\]
where
\[a_k:=\frac{2}{\pi}\int_{-1}^{1}\frac{f(t)T_k(t)}{\sqrt{1-t^2}}dt.\]
The partial sums converge uniformly on $[-1,1]$.

See \cite{Cheb,Cheney,Mason2,Rivlin,Z} for background on Chebyshev polynomial theory.

\begin{remark}
The Chebyshev polynomials satisfy the formula
\begin{equation}\label{form in one}
    T_m(x)T_k(x)=\frac{1}{2}(T_{m+k}(x)+T_{|m-k|}(x))
\end{equation}
which plays a crucial role in computing coefficients of powers of a given finite Chebyshev expansion. To be more precise, given a polynomial $p:=\sum_{k}a_kT_k$, we may use \eqref{form in one} to expand $p^2$ in its Chebyshev expansion, that is, $p^2=\sum_{k}d_kT_k$. It is then possible to compute the coefficients $d_k$ in terms of the coefficients $a_k$, and this approach provides an answer to the question posed above. We thus see that these computations are important at the stage of the approximate reconstruction of a signal, see \cite{Hammond}.
\end{remark}

\subsection{Chebyshev polynomials in \texorpdfstring{$\RR^2$}{R\texttwosuperior}}
Turning to the two-variable setting, we have the following theorems: if $f\in C([-1,1]^2)$ either satisfies a Lipschitz condition, or is of bounded variation in the square $[-1,1]^2$ with one of its partial derivatives
bounded in the square, then $f$ has a uniformly convergent double
Chebyshev series expansion on $[-1,1]^2$ of the form
\begin{equation}\label{DoubleMas}
f(x,y)=\sump_{n=0}^{\infty}\sump_{k=0}^{\infty}a_{n,k}T_n(x)T_k(y).
\end{equation}
See \cite[pp. 702--710]{Hobson2} and \cite{Mason1,Mason2}, Section 5.3.3,
for a discussion of these theorems.

It is then clear that the Chebyshev expansion obtained from \eqref{DoubleMas} may serve as an approximant in the setting of time-varying signals. To be more precise, let \[p(x,y):=\sum_{n,k}a_{n,k}T_n(x)T_k(y)\] and consider its square \[p(x,y)^2=\sum_{n,k}d_{n,k}T_n(x)T_k(y).\]
Applying identity \eqref{form in one} in each variable, it is possible to compute the $d_{n,k}$ in terms of the $a_{n,k}$.

The challenge with using this double Chebyshev expansion is that the resulting formula becomes rather involved. In particular, applying identity \eqref{form in one} separately in each variable leads to exactly nine different cases for $d_{n,k}$. There are three cases for each variable, as in the one-variable scheme recalled in Section~\ref{Back Appl}, and each consists of a combination of multiple terms in the $a_{n,k}$.

Alternatively, let us recall another possible approach for obtaining a suitable formula. Authors studying Chebyshev polynomials have also defined the two-variable polynomials $T_{m,n}(x,y):=p_{m,n}^{-1/2}(x,y)$ in $\RR^2$ as a generalization of the classical Chebyshev polynomials, see \cite{Tsehia,Koor1,Koor2,Chang} for more information. Multivariable Chebyshev polynomials are also of central importance in the study of PDEs. This definition comes from the one-variable polynomials:
\[p_k^{-1/2}(x)=u^k+u^{-k}, \quad x=u+u^{-1},\]
for $k\in \mathbb{Z}$, which are closely related to the Chebyshev polynomials by the formula
\[p_k^{-1/2}(2x)=2T_k(x), \quad k=0,1,\dots.\]

According to \cite{Tsehia}, the polynomials $T_{m,n}$, $m,n\in \mathbb{Z}$, are defined by
\[T_{m,n}(x,y)=(u^m+v^m+w^m)(u^{-n}+v^{-n}+w^{-n})-(u^{m-n}+v^{m-n}+w^{m-n}),\]
where $x=u+v+w$, $y=uv+uw+vw$ and $uvw=1$.

In the discussion above, the key point for computing the coefficients of $p^2$ in terms of the coefficients of $p$ was the relation \eqref{form in one}. Therefore, in order to apply the theory of two-variable Chebyshev polynomials, a relation among them is needed. To obtain such a relation, we need the following formula
\begin{equation}\label{che}
T_{m,n}(x,y)=u^mv^{-n}+u^mw^{-n}+v^mu^{-n}+v^mw^{-n}+w^mu^{-n}+w^mv^{-n},
\end{equation}
which is deduced from the definition.

\begin{lemma}\label{form for 2}
For all $m,n,k,\ell\in \mathbb{Z}$, the two-variable Chebyshev polynomials satisfy
\begin{align*}
T_{m,n}(x,y)T_{k,\ell}(x,y)=&T_{m+k,n+\ell}(x,y)+T_{m-\ell,n-k}(x,y)+T_{m+k+\ell,n-\ell}(x,y)\\
&+T_{k-n+\ell,-m-\ell}(x,y)+T_{m-\ell+n,-k-n}(x,y)+T_{m+n+\ell,-k-n-\ell}(x,y).
\end{align*}
\begin{proof}
See Appendix~\ref{AppA}.
\end{proof}
\end{lemma}

Once more, given a polynomial of the form $p:=\sum_{m,n}a_{m,n}T_{m,n}$ and applying Lemma~\ref{form for 2}, it is possible to expand $p^2=\sum_{m,n}d_{m,n}T_{m,n}$, where the $d_{m,n}$ are computed in terms of the coefficients $a_{m,n}$. As a matter of fact, the resulting formula becomes even more complicated than the one in the previous discussion. In particular, the six-term product formula of Lemma~\ref{form for 2} leads to an even larger number of cases for $d_{m,n}$, each consisting of a combination of complicated multiple terms in the $a_{m,n}$.

\subsection{Main results}
Alternatively, we shall obtain a Chebyshev expansion with time-varying coefficients that reduces the computations. This Chebyshev expansion mirrors the one used for signal reconstruction in \cite{Hammond}; the only difference is that the coefficients now depend on the time variable.

\begin{lemma}\label{sec2}
Let $\mathcal{I}:=\{J\subseteq [-1,1]: J \text{ is an interval}\}$. Then
\[\sigma:=\sup_{K\in \mathbb{N},x\in [-1,1], J\in \mathcal{I}}\{\Big|\sump_{k=0}^{K}\int_{J}\frac{T_k(t)T_k(x)}{\sqrt{1-t^2}}dt\Big|\}<\infty.\]
\end{lemma}
\begin{proof}
Fix $K\in \mathbb{N},$ $x\in [-1,1]$ and $J=[\alpha,\beta]\in\mathcal{I}.$   Set $x:=\cos\varphi,$ $\alpha:=\cos\gamma$ and $\beta:=\cos\omega,$ for  $\varphi,\gamma,\omega\in [0,\pi]$ with $\omega\leq\gamma$.

The general change-of-variable $t=\cos\theta,$ $\theta\in [0,\pi],$ yields
\begin{align*}
\sump_{k=0}^{K}\int_{J}\frac{T_k(t)T_k(x)}{\sqrt{1-t^2}}dt
&=\sump_{k=0}^{K} \int_\omega^\gamma\cos(k\theta)\cos(k\varphi)\,d\theta\\
&= \frac{\gamma-\omega}{2}+\sum_{k=1}^{K} \cos(k\varphi)
\left[\frac{\sin(k\gamma)-\sin(k\omega)}{k}\right].
\end{align*}

By the trigonometrical identity
$$\cos A\sin B=\frac{1}{2}(\sin(A+B)-\sin(A-B)),$$
it follows that
\begin{align*}
\sump_{k=0}^{K}\int_{J}\frac{T_k(t)T_k(x)}{\sqrt{1-t^2}}dt
&=\frac{\gamma-\omega}{2}+\frac{1}{2}\left[\sum_{k=1}^{K}\frac{\sin(k(\varphi+\gamma))}{k}
-\sum_{k=1}^{K}\frac{\sin(k(\varphi-\gamma))}{k}\right]\\
&-\frac{1}{2}\left[\sum_{k=1}^{K}\frac{\sin(k(\varphi+\omega))}{k}
-\sum_{k=1}^{K}\frac{\sin(k(\varphi-\omega))}{k}\right],
\end{align*}
where all the partial sums are uniformly bounded in $K$ and in the variables
$\varphi\pm\gamma,$ and $\varphi\pm\omega,$ see \cite[p.61]{Z}.
\end{proof}

\begin{definition}\label{def var}
Let $f:[-1,1]^2\rightarrow \RR$ and $x\in [-1,1]$. Define the variation of the function $f(x,\cdot)$ by
\[V_{-1}^{1}[f(x,\cdot)]:=\sup_{P}\Big\{\sum_{k=1}^{K}|f(x,t_k)-f(x,t_{k-1})|\Big\},\]
where the supremum is taken over all partitions $P=\{t_0=-1<t_1<\dots<t_K=1\}$ of $[-1,1]$ (similarly for $f(\cdot,y)$, $y\in [-1,1]$). Next, define the term
\[\nu_1(\delta):=\sup_{|x-z|<\delta}\{V_{-1}^{1}[f(x,\cdot)-f(z,\cdot)]\}.\]
Similarly, we define
\[\nu_2(\delta):=\sup_{|y-t|<\delta}\{V_{-1}^{1}[f(\cdot,y)-f(\cdot,t)]\}.\]

We say that $f$ satisfies a variation-type condition in the square $[-1,1]^2$ if
\[\sup_{x\in [-1,1]}V_{-1}^{1}[f(x,\cdot)], \sup_{y\in [-1,1]}V_{-1}^{1}[f(\cdot,y)]<\infty,\]
and
\[\nu(\delta):=\max\{\nu_1(\delta),\nu_2(\delta)\}\rightarrow 0, \text{ as } \delta\rightarrow 0.\]
\end{definition}

See \cite{Adam} for several definitions of bounded variation for functions of two real variables and equivalences between them.

\begin{remark}\label{rem C1}
Note that if $f\in C([-1,1]^2)$ has continuous first-order partial derivatives, then it satisfies a variation-type condition.
\end{remark}

\begin{remark}
In \cite[Theorem 5.9]{Mason2}, the statement assumes that the function is of bounded variation in the square $[-1,1]^2$ and one of its partial derivatives
is bounded in the square. We emphasize that Definition~\ref{def var} makes no reference to the behavior of the derivatives of the function.
\end{remark}

\begin{theorem}\label{arzela}
If $f\in C([-1,1]^2)$ satisfies a variation-type condition, then $f$ has a uniformly convergent series expansion on $[-1,1]^2$ of the form
\[f(x,y)=\sump_{k=0}^{\infty}a_k(y)T_k(x),\]
where
\[a_k(y):=\frac{2}{\pi}\int_{-1}^{1}\frac{f(t,y)T_k(t)}{\sqrt{1-t^2}}dt.\]
The series expansion is called a uniform Chebyshev expansion, and the $K$-th partial sum is denoted by $\mathcal{T}_K(f)$.
\begin{proof}
By assumption, $f(\cdot,y)$ has bounded variation for each fixed $y\in [-1,1]$, so we can write $f$ in its uniform Chebyshev expansion:
\[f(x,y)=\sump_{k=0}^{\infty}a_k(y)T_k(x),\]
for all $x\in [-1,1]$, see \cite[Theorem 5.7]{Mason2}.

The functions
\begin{equation*}
a_k(y)=\frac{2}{\pi}\int_{-1}^{1}\frac{f(t,y)T_k(t)}{\sqrt{1-t^2}}dt, \quad y\in [-1,1],
\end{equation*}
are continuous in $[-1,1]$.

We shall prove that $\{\sump_{k=0}^{K}a_k(y)T_k(x)\}_{K=0}^{\infty}$ is a uniformly equicontinuous family. Since the functions are defined on a compact set, it is enough to show equicontinuity at each point.

Pick a point $(x_0,y_0)\in [-1,1]^2$ and let $\epsilon>0$. By the triangle inequality,
\begin{align}
\Big|\sump_{k=0}^{K}a_k(y_0)T_k(x_0)&-\sump_{k=0}^{K}a_k(y)T_k(x)\Big|\nonumber\\
&\leq\Big|\sump_{k=0}^{K}[a_k(y_0)-a_k(y)]T_k(x)\Big|\label{expa1}\\
&+\Big|\sump_{k=0}^{K}a_k(y_0)[T_k(x_0)-T_k(x)]\Big|.\nonumber
\end{align}

By the uniform Chebyshev expansion of the function $f(\cdot,y_0)$ (recall that a uniformly convergent sequence of continuous functions on a compact set is equicontinuous), there exists $\delta_1>0$ such that
\begin{equation}\label{A}
\Big|\sump_{k=0}^{K}a_k(y_0)[T_k(x_0)-T_k(x)]\Big|<\epsilon/2,
\end{equation}
for all $x\in (x_0-\delta_1,x_0+\delta_1)$ and $K\in \mathbb N$.

Set $g_{y,y_0}(\cdot):=f(\cdot,y_0)-f(\cdot,y)$. We now estimate the first term on the right-hand side of \eqref{expa1}: by the variation mean value theorem, see \cite[p.~570]{Hobson1}, we obtain
\begin{align*}
&\Big|\sump_{k=0}^{K}[a_k(y_0)-a_k(y)]T_k(x)\Big|
=\Big|\sump_{k=0}^{K}\frac{2}{\pi}\int_{-1}^{1}\frac{(f(t,y_0)-f(t,y))T_k(t)T_k(x)}{\sqrt{1-t^2}}dt\Big|\\
&=\frac{2}{\pi}\Big|\int_{-1}^{1}\frac{(f(t,y_0)-f(t,y))\sump_{k=0}^{K}T_k(t)T_k(x)}{\sqrt{1-t^2}}dt\Big|\\
&\leq \frac{2}{\pi}(V_{-1}^{1}[g_{y,y_0}(\cdot)]+|g_{y,y_0}(-1)|)\sup_{J\in \mathcal{I}}\Big|\int_{J}\frac{\sump_{k=0}^{K}T_k(t)T_k(x)}{\sqrt{1-t^2}}dt\Big|\\
&\leq \frac{2}{\pi}(V_{-1}^{1}[g_{y,y_0}(\cdot)]+|g_{y,y_0}(-1)|)\sup_{K\in \mathbb N, x\in[-1,1], J\in \mathcal{I}}\Big|\int_{J}\frac{\sump_{k=0}^{K }T_k(t)T_k(x)}{\sqrt{1-t^2}}dt\Big|.
\end{align*}

Thus, for $\delta>0$ and $y\in (y_0-\delta,y_0+\delta)$, we get
\begin{equation*}
\Big|\sump_{k=0}^{K}[a_k(y_0)-a_k(y)]T_k(x)\Big|\leq
\frac{2}{\pi}\sigma\,(\nu_2(\delta)+|f(-1,y_0)-f(-1,y)|)\rightarrow 0, \text{ as } \delta\rightarrow 0,
\end{equation*}
hence, there exists $\delta_2>0$ such that, for all $y\in(y_0-\delta_2,y_0+\delta_2)$,
\begin{equation}\label{B}
\Big|\sump_{k=0}^{K}[a_k(y_0)-a_k(y)]T_k(x)\Big|<\epsilon/2.
\end{equation}

It follows from \eqref{A} and \eqref{B} that
\begin{align*}
 \Big|\sump_{k=0}^{K}a_k(y_0)T_k(x_0)&-\sump_{k=0}^{K}a_k(y)T_k(x)\Big|<\epsilon,
\end{align*}
for all $(x,y)$ with $\lVert (x,y)-(x_0,y_0)\rVert<\min\{\delta_1,\delta_2\}$ and all $K\in \mathbb N$.

The family is equicontinuous and converges pointwise to the function $f$ on the compact set $[-1,1]^2$. We infer that the convergence is uniform, see \cite[Chapter 7, Exercise 16]{Rudin}.
\end{proof}
\end{theorem}

From Theorem~\ref{arzela} we may derive a uniform double Chebyshev series expansion for functions satisfying an additional property.

Let us recall that, in the one-variable context, a sequence $\{f_n\}$ of functions with bounded variation on $[-1,1]$ is said to converge to $f$ in the variation seminorm if $V_{-1}^{1}[f_n-f]\rightarrow 0$ as $n\rightarrow\infty$ (note that $V_{-1}^{1}[\cdot]$ is a seminorm, as it vanishes on constant functions). Although convergence in the variation seminorm is a strong requirement, there exist sequences of convex functions that converge in the variation seminorm to a continuous function and may be chosen neither to be Dini–Lipschitz continuous nor to have bounded derivatives.

We may consider a similar notion of convergence for functions satisfying a variation-type condition.
\begin{definition}\label{def bv}
We say that a sequence $\{f_n\}$ of functions satisfying a variation-type condition on $[-1,1]^2$ converges to $f$ in the sense of bounded variation if one of the following conditions is satisfied:
\begin{enumerate}[(i)]
\item $\sup_{x\in[-1,1]}V_{-1}^{1}[f_n(x,\cdot)-f(x,\cdot)]\rightarrow 0$, as $n\rightarrow \infty$;
\item $\sup_{y\in[-1,1]}V_{-1}^{1}[f_n(\cdot,y)-f(\cdot,y)]\rightarrow 0$, as $n\rightarrow \infty$.
\end{enumerate}
\end{definition}

With this definition in hand, we may obtain a counterpart of \cite[Theorems 5.9 and 5.10]{Mason2}. Let us note that the partial sums of the Chebyshev expansion in Theorem~\ref{arzela} satisfy a variation-type condition on $[-1,1]^2$.

\begin{theorem}\label{unif in 2}
Let $f\in C([-1,1]^2)$ satisfy a variation-type condition, and consider the partial sums
\[\mathcal{T}_K(t,y):=\sump_{k=0}^{K}b_k(t)T_k(y), \quad \text{ where } \quad b_k(t):=\frac{2}{\pi}\int_{-1}^{1}\frac{f(t,s)T_k(s)}{\sqrt{1-s^2}}ds.\]
Suppose that
\begin{equation}\label{BVhyp}
\sup_{y\in[-1,1]}V_{-1}^{1}\big[\mathcal{T}_K(\cdot,y)-f(\cdot,y)\big]\rightarrow 0, \quad \text{ as } K\rightarrow \infty,
\end{equation}
namely, that the sequence $\{\mathcal{T}_K\}$ converges to $f$ in the sense of bounded variation through condition (ii) of Definition~\ref{def bv}. Then $f$ has a uniformly convergent Chebyshev series expansion on $[-1,1]^2$ of the form
\[f(x,y)=\sump_{n=0}^{\infty}\sump_{k=0}^{\infty}a_{n,k}T_n(x)T_k(y),\]
where
\[a_{n,k}:=\frac{4}{\pi^2}\int_{-1}^{1}\int_{-1}^{1}\frac{f(t,s)T_n(t)T_k(s)}{\sqrt{1-t^2}\sqrt{1-s^2}}dtds.\]
A symmetric statement holds with the roles of the two variables interchanged. The family $\{T_n(x)T_k(y)\}_{n,k}$ forms an orthogonal system in $L^2\big([-1,1]^2,\frac{dt\,ds}{\sqrt{1-t^2}\sqrt{1-s^2}}\big)$.
\begin{proof}
Orthogonality follows from the arguments in the one-dimensional setting, see \cite{Mason2}. By Theorem~\ref{arzela}, applied with the roles of the variables interchanged (note that the variation-type condition is symmetric in the two variables), the partial sums $\mathcal{T}_K$ converge uniformly to $f$ on $[-1,1]^2$, that is,
\begin{equation}\label{pos}
f(t,y)=\sump_{k=0}^{\infty}b_k(t)T_k(y),
\end{equation}
uniformly on $[-1,1]^2$.

By straightforward computations we obtain for each $K,L\in \mathbb{N}$,
\begin{equation}\label{AAA}
\sump_{n=0}^{L}\sump_{k=0}^{K}a_{n,k}T_n(x)T_k(y)=\frac{2}{\pi}\int_{-1}^{1}\mathcal{T}_K(t,y)\frac{\sump_{n=0}^{L}T_n(t)T_n(x)}{\sqrt{1-t^2}}dt.
\end{equation}

We shall prove that the right-hand side in \eqref{AAA} tends to
\[
\frac{2}{\pi}\int_{-1}^{1}f(t,y)\frac{\sump_{n=0}^{L}T_n(t)T_n(x)}{\sqrt{1-t^2}}dt
\]
as $K\rightarrow \infty$, uniformly with respect to $x,y\in [-1,1]$ and $L\in \mathbb N$.

Indeed, fix $y\in [-1,1]$. Applying the variation mean value theorem of \cite[p.~570]{Hobson1}, we obtain
\begin{align*}
\Big|\int_{-1}^{1}&[\mathcal{T}_K(t,y)-f(t,y)]\frac{\sump_{n=0}^{L}T_n(t)T_n(x)}{\sqrt{1-t^2}}dt\Big|\\
&\leq \sigma\, \Big(V_{-1}^{1}\big[ \mathcal{T}_K(\cdot,y)-f(\cdot,y)\big]+\big|\mathcal{T}_K(-1,y)-f(-1,y)\big|\Big).
\end{align*}
By hypothesis \eqref{BVhyp},
\[\sup_{y\in [-1,1]}V_{-1}^{1}\big[ \mathcal{T}_K(\cdot,y)-f(\cdot,y)\big]\rightarrow 0, \quad \text{ as } K\rightarrow \infty.\]
Moreover, \eqref{pos} yields
\[\sup_{y\in [-1,1]}\big|\mathcal{T}_K(-1,y)-f(-1,y)\big|\rightarrow 0, \quad \text{ as } K\rightarrow \infty,\]
and hence the proof of the assertion made about the uniform convergence of the RHS of \eqref{AAA} is complete.

Furthermore,
\begin{align*}
\frac{2}{\pi}\int_{-1}^{1}f(t,y)\frac{\sump_{n=0}^{L}T_n(t)T_n(x)}{\sqrt{1-t^2}}dt&=\frac{2}{\pi}\sump_{n=0}^{L} T_n(x)\int_{-1}^{1}f(t,y)\frac{T_n(t)}{\sqrt{1-t^2}}dt\\
&=\sump_{n=0}^{L}a_n(y)T_n(x),
\end{align*}
where $a_n(y)$ are the coefficients of Theorem~\ref{arzela}, and the last sum converges uniformly to $f$ on the square $[-1,1]^2$, as $L\rightarrow \infty$, by Theorem~\ref{arzela}.

The triangle inequality and the uniform convergences obtained above yield the desired result.
\end{proof}
\end{theorem}

Although our focus is on applications of the approximation results, it is interesting from a theoretical standpoint to examine how the assumptions in the approximation theorems of Section 5.3.3 of \cite{Mason2} and those of Theorem~\ref{arzela} and Theorem~\ref{unif in 2} imply one another. To be more precise, a question that arises is whether, or under which additional conditions, the Chebyshev expansion of a function $f$ satisfying the assumptions of \cite[Theorem 5.9]{Mason2} converges to $f$ in the sense of bounded variation.

Let us also note that approximation results such as Theorem~\ref{arzela} may be useful in regression analysis (see Subsection 3.2 of \cite{Viv}, which is a numerical experiment based on \cite{Breiman}), in recovering the original image from a noisy one in the two-dimensional continuous framework, where the space of BV functions is an appropriate function class for many basic image processing tasks, see \cite{RudinL2}, and also in PDE-based image processing, see \cite{Buades,RudinL1}.

From Theorem~\ref{unif in 2} we can now deduce the following uniform Chebyshev expansion. See Section 5.3.3 of \cite{Mason2} for counterparts that also apply to more general functions.

\begin{corollary}
Let $f\in C([-1,1]^2)$ satisfy a variation-type condition and assume that one of its partial derivatives is continuous and also satisfies a variation-type condition. Then $f$ has a uniformly convergent Chebyshev series expansion on $[-1,1]^2$ of the form
\[f(x,y)=\sump_{n=0}^{\infty}\sump_{k=0}^{\infty}a_{n,k}T_n(x)T_k(y),\]
where
\[a_{n,k}:=\frac{4}{\pi^2}\int_{-1}^{1}\int_{-1}^{1}\frac{f(t,s)T_n(t)T_k(s)}{\sqrt{1-t^2}\sqrt{1-s^2}}dtds.\]
In particular, every $C^2$ function has a uniform double series Chebyshev expansion.
\begin{proof}
Assume that $\partial_x f$ is continuous on $[-1,1]^2$ and satisfies a variation-type condition; the case of $\partial_y f$ is symmetric, using the symmetric version of Theorem~\ref{unif in 2}. Let $\mathcal{T}_K$ and $b_k$ be as in Theorem~\ref{unif in 2}. Differentiating under the integral sign, we obtain
\[\partial_t \mathcal{T}_K(t,y)=\sump_{k=0}^{K}b'_k(t)T_k(y), \quad \text{ where } \quad b'_k(t)=\frac{2}{\pi}\int_{-1}^{1}\frac{\partial_t f(t,s)T_k(s)}{\sqrt{1-s^2}}ds,\]
that is, $\partial_t \mathcal{T}_K$ coincides with the corresponding partial sum of $\partial_x f$. By Theorem~\ref{arzela}, applied to $\partial_x f$ with the roles of the variables interchanged, $\partial_t \mathcal{T}_K\rightarrow \partial_x f$ uniformly on $[-1,1]^2$. Since $\mathcal{T}_K(\cdot,y)-f(\cdot,y)$ is continuously differentiable for each $y\in [-1,1]$, we deduce that
\[\sup_{y\in [-1,1]}V_{-1}^{1}\big[\mathcal{T}_K(\cdot,y)-f(\cdot,y)\big]\leq 2\sup_{(t,y)\in [-1,1]^2}\big|\partial_t \mathcal{T}_K(t,y)-\partial_x f(t,y)\big|\rightarrow 0, \quad \text{ as } K\rightarrow \infty,\]
where we used that $V_{-1}^{1}[u]\leq 2\sup_{[-1,1]}|u'|$ for continuously differentiable $u$. Thus hypothesis \eqref{BVhyp} of Theorem~\ref{unif in 2} is satisfied, and the conclusion follows. Finally, if $f$ is of class $C^2$, then $f$ and $\partial_x f$ have continuous first-order partial derivatives, hence they satisfy a variation-type condition by Remark~\ref{rem C1}.
\end{proof}
\end{corollary}

\begin{remark}
Note that, unlike the uniform expansion coefficients in Theorem~\ref{arzela}, the uniform expansion coefficients in Theorem~\ref{unif in 2} are constant. Therefore, we get two approaches to applying approximation theory to the reconstruction of time-varying signals.

However, in both cases, it is necessary to compute an integral with respect to a specific measure whose integrand is the product of the scaling/wavelet functions and Chebyshev polynomials. The difference lies in the fact that the use of double-series Chebyshev expansions leads to a more challenging formula, yet at the same time it shows that, for sufficiently regular functions, the canonical rectangular partial sums of the double Chebyshev expansion converge uniformly. (Let us point out that mere membership in the uniform closure of $\operatorname{span}\{T_n(x)T_k(y):n,k=0,1,2,\dots\}$ holds for every function in $C([-1,1]^2)$ by the Stone--Weierstrass theorem.)

On the other hand, the use of the Chebyshev expansion in Theorem~\ref{arzela} leads to an easier formula consisting of time-varying coefficients. Given scaling and wavelet functions, one can derive the exact form of the approximant composition $\tilde{W}^*[\tau]\tilde{W}[\tau]\ff(\tau)$ for all time points $\tau$.

Scaling functions such as $e^{-sF(x,\tau)}$, with $F$ very regular, play a fundamental role in various areas of applied analysis. Also, there exist non-stationary kernel functions that are of central importance in regression analysis, see, for example, \cite{Paciorek2}.
\end{remark}

It is reasonable to inquire about best polynomial approximants in Remark~\ref{approx}. As we shall see in the next section, we can transfer the approximation results from $[-1,1]^2$ to $[0,\lambda_{\text{max}}]\times[0,\tau^*]$.

Let $P_{K,L}$ be the finite-dimensional space of polynomials of degree at most $K$ in $x$ and at most $L$ in $y$. A polynomial $p_{K,L}^*\in P_{K,L}$ is called a best polynomial approximant to $g\in C([-1,1]^2)$ if the error (with respect to the supremum norm)
\[S_{K,L,\infty}:=\sup_{x,y\in[-1,1]}\{|g(x,y)-p^*_{K,L}(x,y)|\}\]
satisfies
\[S_{K,L,\infty}\leq \sup_{x,y\in[-1,1]}\{|g(x,y)-p(x,y)|\}, \quad \text{ for all } p\in P_{K,L}.\]
See \cite{Cheney} about the existence of best polynomial approximations.

It is also of fundamental interest to consider forms of approximation other than polynomials for which the existence or uniqueness of a best approximation holds, see \cite{Cheney}.

Let us also note that the partial sum $\mathcal{T}_K(g)(x,y):=\sump_{k=0}^{K}a_k(y)T_k(x)$, with coefficients $a_k(y)$ defined as in Theorem~\ref{arzela}, is well defined for any $g\in C([-1,1]^2)$, regardless of whether $g$ satisfies a variation-type condition.

\begin{theorem}\label{general error}
If $g\in C([-1,1]^2)$, then
\[\sup_{x,y\in [-1,1]}\{|g(x,y)-\mathcal{T}_K(g)(x,y)|\} \leq (4+\frac{4}{\pi^2}\log K)S_{K,L,\infty},\]
for all integers $K\geq 1$ and all $L\in \mathbb N$.
\begin{proof}
Indeed, fix $y\in [-1,1]$, an integer $K\geq 1$, and $L\in \mathbb N$, and let $p^*_{y,K}$ denote the best polynomial approximant of degree at most $K$ to $g(\cdot,y)$ on $[-1,1]$. Theorem 3.3, p.~134 of \cite{Rivlin}, together with the definition of $\mathcal{T}_K(g)$, yields
\begin{align*}
\sup_{x\in [-1,1]}\{|g(x,y)-\mathcal{T}_K(g)(x,y)|\}&\leq (4+\frac{4}{\pi^2}\log K)\sup_{x\in [-1,1]}\{|g(x,y)-p^*_{y,K}(x)|\}\\
&\leq (4+\frac{4}{\pi^2}\log K)\sup_{x\in [-1,1]}\{|g(x,y)-p^*_{K,L}(x,y)|\},
\end{align*}
from which it follows that
\begin{equation*}
 \sup_{x,y\in [-1,1]}\{|g(x,y)-\mathcal{T}_K(g)(x,y)|\} \leq (4+\frac{4}{\pi^2}\log K)S_{K,L,\infty},
\end{equation*}
for all $K\geq 1$ and $L\in \mathbb N$. This gives the assertion.
\end{proof}
\end{theorem}

Let us note that the approximant function $\mathcal{T}_K(g)$ is not a polynomial in both variables; it is a finite linear combination of Chebyshev polynomials $T_k(x)$, where the coefficients $a_k=a_k(y)$ depend on the variable $y\in [-1,1]$. Theorem~\ref{general error} says that, for an arbitrary continuous function $g$, the approximant $\mathcal{T}_K(g)$ yields a near-best uniform approximation, up to a logarithmically growing factor.

On the other hand, we have the following error estimate for functions that are $(m+1)$-times continuously differentiable with respect to $x$.

\begin{theorem}\label{K-error}
Let $m\geq 1$. If $g\in C([-1,1]^2)$ is $(m+1)$-times continuously differentiable with respect to $x$, that is, the partial derivatives $\partial_{x}^{j}g$, $j=1,\dots,m+1$, exist and are continuous on $[-1,1]^2$, then there exists a constant $C>0$, which does not depend on $K$, such that
\begin{equation}
\sup_{x,y\in [-1,1]}\{|g(x,y)-\mathcal{T}_{K}(g)(x,y)|\}\leq \frac{C}{K^m},
\end{equation}
for all $K\geq 1$.
\begin{proof}
Set $g_y:=g(\cdot,y)$, for fixed $y\in [-1,1]$, and note that $\{g_y\}_{y\in [-1,1]}\subset C^{m+1}([-1,1])$. We may apply the arguments in the proof of \cite[Theorem 5.14]{Mason2} to the family $\{g_y\}_{y\in [-1,1]}$. Indeed, applying Peano's representation theorem, for $K\geq m$, we obtain
\begin{align*}
|g(x,y)-\mathcal{T}_{K}(g)(x,y)|&=\Big|\int_{-1}^{1}g_y^{(m+1)}(t)\mathrm{Ker}_{K}(x,t)dt\Big|\\
&\leq \sup_{t\in [-1,1]}\{|g_y^{(m+1)}(t)|\}\int_{-1}^{1}|\mathrm{Ker}_{K}(x,t)|dt\\
&\leq \sup_{x,y\in [-1,1]}\{|\partial_{x}^{m+1}g(x,y)|\}\,C_m\,K^{-m},
\end{align*}
where, as in the proof of \cite[Theorem 5.14]{Mason2}, the Peano kernel satisfies the estimate
\[\sup_{x\in [-1,1]}\int_{-1}^{1}|\mathrm{Ker}_{K}(x,t)|dt\leq C_m\,K^{-m}, \quad K\geq m,\]
with a constant $C_m>0$ depending only on $m$ (in particular, the bound is uniform in $x$ and $y$), and the supremum of $|\partial_{x}^{m+1}g|$ is finite since $\partial_{x}^{m+1}g$ is continuous on the compact set $[-1,1]^2$. The finitely many values $1\leq K<m$ are covered by enlarging the constant, since $\sup_{x,y\in [-1,1]}\{|g(x,y)-\mathcal{T}_{K}(g)(x,y)|\}$ is finite for each fixed $K$. From this it follows that
\begin{equation*}
\sup_{x,y\in [-1,1]}\{|g(x,y)-\mathcal{T}_{K}(g)(x,y)|\}\leq \frac{C}{K^m},
\end{equation*}
for all $K\geq 1$, where the constant $C>0$ does not depend on $K$.
\end{proof}
\end{theorem}

Theorem~\ref{K-error} provides a uniform error bound for $\mathcal{T}_K(g)$ as an approximant to a function $g\in C([-1,1]^2)$ which is $(m+1)$-times continuously differentiable with respect to $x$.

\section{Application of the uniform Chebyshev expansion to the approximate reconstruction of a time-varying signal}\label{Back Appl}

Our goal is to obtain a uniform composition approximant, namely $\tilde{W}^*[\tau]\tilde{W}[\tau]\ff(\tau)$, analogous to the one in Subsection 6.1 of \cite{Hammond}. In this setting, we shall apply Theorem~\ref{arzela}, and the time-varying coefficients will contain all the relevant information. The reconstruction of the time-varying signal may be derived from the discussion presented in Section 7 of \cite{Hammond}; in Proposition~\ref{stability} below we further prove that the resulting reconstruction is stable.

Fix a scale $s>0$ and let $g(sx,\tau)$ be defined for $x\in [0,\lambda^*]$ and $\tau\in [0,\tau^*]$, where $\lambda^*\geq\lambda_{N-1}$. We may shift the domain $[0,\lambda^*]\times [0,\tau^*]$ to $[-1,1]^2$ by applying the transformation $\xi=(2x-\lambda^*)/\lambda^*$ and $z=(2\tau-\tau^*)/\tau^*$. If the function
$g(s\frac{\lambda^*}{2}(\xi+1),\frac{\tau^*}{2}(z+1))$, $(\xi,z)\in [-1,1]^2$, is continuous and satisfies a variation-type condition, then, by Theorem~\ref{arzela}, we obtain
\[ g(s\frac{\lambda^*}{2}(\xi+1),\frac{\tau^*}{2}(z+1))=\sump_{k=0}^{\infty}a^s_k(z)T_k(\xi),\]
that is,
\[ g(sx,\tau)=\sump_{k=0}^{\infty}a^s_k\left(\frac{2\tau-\tau^*}{\tau^*}\right)T_k\left(\frac{2x-\lambda^*}{\lambda^*}\right),\quad  x\in [0,\lambda^*],\tau\in [0,\tau^*],\]
where the partial sums converge uniformly on $[0,\lambda^*]\times [0,\tau^*]$.

Let us suppose that $h(x,\tau),g(s_1x,\tau),\dots,g(s_rx,\tau)$ are continuous and, after the above change of variables, satisfy a variation-type condition on $[-1,1]^2$. Set
\[\overline{a}_{j,k}(\tau):=a^{s_j}_k\left(\frac{2\tau-\tau^*}{\tau^*}\right)=\frac{2}{\pi}\int_{-1}^{1}\frac{g(s_j\frac{\lambda^*}{2}(t+1),\tau)T_k(t)}{\sqrt{1-t^2}}dt, \quad j=1,\dots,r,\]
\[\overline{a}_{0,k}(\tau):=\frac{2}{\pi}\int_{-1}^{1}\frac{h(\frac{\lambda^*}{2}(t+1),\tau)T_k(t)}{\sqrt{1-t^2}}dt\]
and
\[ \overline{T}_k(x):=T_k\left(\frac{2x-\lambda^*}{\lambda^*}\right).\]
Then
\[ h(x,\tau )=\sump_{k=0}^{\infty}\overline{a}_{0,k}(\tau)\overline{T}_k(x), \quad  x\in [0,\lambda^*],\tau\in [0,\tau^*],\]
\[ g(s_jx,\tau )=\sump_{k=0}^{\infty}\overline{a}_{j,k}(\tau)\overline{T}_k(x), \quad x\in [0,\lambda^*],\tau\in [0,\tau^*],\]
where the partial sums converge uniformly on $[0,\lambda^*]\times [0,\tau^*]$.

Let
\[p_j(x,\tau)=\sump_{k=0}^{M_j}\overline{a}_{j,k}(\tau)\overline{T}_k(x), \quad x\in [0,\lambda^*],\tau\in [0,\tau^*],\; j=0,1,\dots,r,\]
be the Chebyshev expansion approximant of order $M_j$ to each $h(x,\tau),g(s_1x,\tau),\dots,g(s_rx,\tau)$, respectively; the degrees $M_j$ may be chosen independently of $\tau$ since the approximation is uniform.

Then, for a given time-varying signal $\ff(\tau)$, we may define the overall approximant transformation \[\tilde{W}[\tau]\ff(\tau)=\Big((T_{p_0}[1,\tau]\ff(\tau))^\top,(T_{p_1}[s_1,\tau]\ff(\tau))^\top,\dots,(T_{p_r}[s_r,\tau]\ff(\tau))^\top\Big)^\top,\]
which, according to Remark~\ref{approx}, approximates $W[\tau]\ff(\tau)$ given in \eqref{GWT}. Denote its adjoint by $\tilde{W}^*[\tau]$.

The computation in Remark~\ref{compos overall}, applied to the kernels $p_j$, yields
\[ \tilde{W}^*[\tau]\tilde{W}[\tau]\ff(\tau)=\Big(\sum_{j=0}^rp_j(\mathcal{L},\tau)^2\Big)\ff(\tau),\]
and hence we need to compute each $p_j(\mathcal{L},\tau)^2$, $j=0,\dots,r$.

To get rid of the $1/2$ factor in the expansion of $p_j$, set $\overline{a}'_{j,0}(\tau)=\frac{1}{2}\overline{a}_{j,0}(\tau)$, and $\overline{a}'_{j,k}(\tau)=\overline{a}_{j,k}(\tau)$, $k\geq 1$,  so that
\[ p_j(x,\tau)=\sum_{k=0}^{M_j}\overline{a}_{j,k}'(\tau)\overline{T}_k(x).\]

We follow the scheme of Subsection 6.1 of \cite{Hammond}. Set
\[
\overline{d}'_{j,k}(\tau)=
\begin{cases}
\frac{1}{2}(\overline{a}_{j,0}'^2(\tau)+\sum_{m=0}^{M_j}\overline{a}_{j,m}'^2(\tau)),  &\text{ if }k=0,\\ \\
\frac{1}{2}(\sum_{m=0}^{k}\overline{a}_{j,m}'(\tau)\overline{a}_{j,k-m}'(\tau)+\sum_{m=0}^{M_j-k}\overline{a}_{j,m}'(\tau)\overline{a}_{j,k+m}'(\tau)+& \\ \\
\quad+\sum_{m=k}^{M_j}\overline{a}_{j,m}'(\tau)\overline{a}_{j,m-k}'(\tau)),
&\text{ if }0<k\leq M_j,\\ \\
\frac{1}{2}(\sum_{m=k-M_j}^{M_j}\overline{a}_{j,m}'(\tau)\overline{a}_{j,k-m}'(\tau)),& \text{ if }M_j<k\leq 2M_j,
\end{cases}
\]
and
\[\overline d_{j,0}(\tau)=2\overline d'_{j,0}(\tau), \quad \overline d_{j,k}(\tau)=\overline d'_{j,k}(\tau), \quad k\geq 1, \quad \overline{d}_k(\tau)=\sum_{j=0}^r\overline d_{j,k}(\tau),\]
with the convention that $\overline d'_{j,k}(\tau):=0$ for $k>2M_j$.
Put $M^*:=2\max_{0\leq j\leq r} M_j$. Then the approximant composition at each time point $\tau\in [0,\tau^*]$ is given by
\[\tilde{W}^*[\tau]\tilde{W}[\tau]\ff(\tau)=\sump_{k=0}^{M^*}\overline{d}_k(\tau)\overline T_k(\mathcal{L})\ff(\tau),\]
where $M^*$ does not depend on the time variable $\tau$. In particular, each $\overline{d}_k(\tau)$ is a finite sum of products of two integrals of the form
\[\int_{-1}^{1}\frac{h(\frac{\lambda^*}{2}(t+1),\tau)T_m(t)}{\sqrt{1-t^2}}dt \quad \text{ and } \quad \int_{-1}^{1}\frac{g(s_j\frac{\lambda^*}{2}(t+1),\tau)T_m(t)}{\sqrt{1-t^2}}dt, \quad m\in \mathbb N,\; j=1,\dots,r.\]

We conclude this section by proving that $\tilde{W}[\tau]$ is a small perturbation of $W[\tau]$ and that the reconstruction by means of the approximated transform is stable, provided the uniform approximation errors of the kernels are small enough.

\begin{proposition}\label{stability}
Set
\[\varepsilon_0:=\sup_{x\in [0,\lambda^*],\,\tau\in [0,\tau^*]}|h(x,\tau)-p_0(x,\tau)|, \qquad \varepsilon_j:=\sup_{x\in [0,\lambda^*],\,\tau\in [0,\tau^*]}|g(s_jx,\tau)-p_j(x,\tau)|,\]
$j=1,\dots,r$, and $\varepsilon:=\big(\sum_{j=0}^r\varepsilon_j^2\big)^{1/2}$. Then
\[\lVert \tilde{W}[\tau]-W[\tau]\rVert\leq \varepsilon, \quad \text{ for all } \tau\in [0,\tau^*],\]
where $\lVert \cdot\rVert$ denotes the operator norm. Moreover, assume that
\[A:=\inf_{\tau\in [0,\tau^*]}\min_{0\leq \ell\leq N-1}\Big(h(\lambda_\ell,\tau)^2+\sum_{j=1}^rg(s_j\lambda_\ell,\tau)^2\Big)>0\]
and that $\varepsilon<\sqrt{A}$. Then, for all $\tau\in [0,\tau^*]$ and $\ff(\tau)\in \RR^N$,
\[\lVert \tilde{W}[\tau]\ff(\tau)\rVert\geq (\sqrt{A}-\varepsilon)\,\lVert \ff(\tau)\rVert.\]
In particular, $\tilde{W}^*[\tau]\tilde{W}[\tau]$ is invertible, and the reconstruction by means of the pseudoinverse $\tilde{W}[\tau]^{\dagger}:=(\tilde{W}^*[\tau]\tilde{W}[\tau])^{-1}\tilde{W}^*[\tau]$ is stable:
\[\big\lVert \tilde{W}[\tau]^{\dagger}W[\tau]\ff(\tau)-\ff(\tau)\big\rVert\leq \frac{\varepsilon}{\sqrt{A}-\varepsilon}\,\lVert \ff(\tau)\rVert.\]
\begin{proof}
Since $U$ is orthogonal and $\lambda_\ell\in [0,\lambda^*]$ for all $\ell$, for every $\ff(\tau)\in \RR^N$ we have
\[\lVert T_{p_j}[s_j,\tau]\ff(\tau)-T_g[s_j,\tau]\ff(\tau)\rVert^2=\sum_{\ell=0}^{N-1}\big(p_j(\lambda_\ell,\tau)-g(s_j\lambda_\ell,\tau)\big)^2\hat{f}(\ell,\tau)^2\leq \varepsilon_j^2\,\lVert \ff(\tau)\rVert^2,\]
for $j=1,\dots,r$, and similarly for the scaling block, with $\varepsilon_0$. Summing over the $r+1$ blocks of $\tilde{W}[\tau]\ff(\tau)-W[\tau]\ff(\tau)$ we obtain
\[\lVert \tilde{W}[\tau]\ff(\tau)-W[\tau]\ff(\tau)\rVert^2\leq \Big(\sum_{j=0}^{r}\varepsilon_j^2\Big)\lVert \ff(\tau)\rVert^2=\varepsilon^2\,\lVert \ff(\tau)\rVert^2,\]
which proves the first claim. Next, arguing as in Remark~\ref{compos overall},
\begin{align*}
\lVert W[\tau]\ff(\tau)\rVert^2&=\langle W^*[\tau]W[\tau]\ff(\tau),\ff(\tau)\rangle\\
&=\sum_{\ell=0}^{N-1}\Big(h(\lambda_\ell,\tau)^2+\sum_{j=1}^rg(s_j\lambda_\ell,\tau)^2\Big)\hat{f}(\ell,\tau)^2\geq A\,\lVert \ff(\tau)\rVert^2,
\end{align*}
and hence, by the triangle inequality,
\[\lVert \tilde{W}[\tau]\ff(\tau)\rVert\geq \lVert W[\tau]\ff(\tau)\rVert-\lVert \tilde{W}[\tau]\ff(\tau)-W[\tau]\ff(\tau)\rVert\geq (\sqrt{A}-\varepsilon)\,\lVert \ff(\tau)\rVert.\]
It follows that the smallest singular value of $\tilde{W}[\tau]$ is at least $\sqrt{A}-\varepsilon>0$; in particular, $\tilde{W}^*[\tau]\tilde{W}[\tau]\succeq (\sqrt{A}-\varepsilon)^2 I_N$ is invertible, $\tilde{W}[\tau]^{\dagger}\tilde{W}[\tau]=I_N$, and $\lVert \tilde{W}[\tau]^{\dagger}\rVert\leq (\sqrt{A}-\varepsilon)^{-1}$. Therefore,
\[\tilde{W}[\tau]^{\dagger}W[\tau]\ff(\tau)-\ff(\tau)=\tilde{W}[\tau]^{\dagger}\big(W[\tau]-\tilde{W}[\tau]\big)\ff(\tau),\]
and the last claim follows from the first one.
\end{proof}
\end{proposition}

\begin{remark}
By the uniform convergence of the Chebyshev expansions established above, $\varepsilon_j\rightarrow 0$ as $M_j\rightarrow \infty$; in particular, the condition $\varepsilon<\sqrt{A}$ is fulfilled provided the degrees $M_j$ are chosen large enough. Moreover, $A>0$ holds under the frame assumptions of Section~\ref{Back SGT}. Proposition~\ref{stability} thus justifies the stable reconstruction of the time-varying signal from its wavelet coefficients by means of the approximated transform.
\end{remark}

\section{Numerical experiments}\label{sec:numerical}

Graph Laplacian techniques have already been used in several image restoration problems and graph signal inverse problems, see for example~\cite{bianchi2021graph,aleotti2026iterated,bianchi2025data} and references therein. This section applies the complete graph wavelet transform to a denoising
experiment.   The sensor graph and the denoising setting are adapted from
the numerical experiment in \cite[Section~IV-D]{ShumanChebyshev2018}, where a single
function of the graph Laplacian, approximated by Chebyshev polynomials,
acts as a linear smoothing operator.  Here the transform is time varying: following the classical framework of
\cite{Donoho1995}, we apply soft thresholding to the graph wavelet
coefficients and reconstruct by solving the associated normal equations.

The test problem is described in
Subsection~\ref{subsec:numerical-problem}, and the experiment then proceeds
in three steps.  First, in Subsection~\ref{subsec:numerical-stability}, we
verify the uniform approximation and stability estimates of
Proposition~\ref{stability}.  Next, in
Subsection~\ref{subsec:sure-shrinkage}, we define the soft-thresholding
reconstruction and select the transform and threshold parameters from the
noisy signals alone by Stein's unbiased risk estimate (SURE)
\cite{Stein1981}.  Finally, holding those parameters fixed, we replace the
exact transform by degree-$K$ Chebyshev approximants, compare the
time-varying and time-independent reconstructions at every tested degree,
and measure convergence of the time-varying reconstruction to its exact
counterpart in
Subsection~\ref{subsec:numerical-polynomial-shrinkage}.

\subsection{Graph, time-varying signal, and noise model}
\label{subsec:numerical-problem}

We draw $N=500$ points $z_n=(x_n,y_n)$ uniformly in the unit square.  Two vertices are joined when their Euclidean distance
$\rho_{ij}$ is at most $7.50\times10^{-2}$, and the corresponding edge weight is
\[
 w_{ij}=\exp\!\left(
 -\frac{\rho_{ij}^2}{2\cdot(7.40\times10^{-2})^2}
 \right).
\]
The resulting connected graph has $|E|=2034$ edges.  We use the easily computable
Anderson--Morley bound
\[
 \lambda^*:=\max_{i\sim j}(d_{ii}+d_{jj})=25.17
 \geq \lambda_{N-1}=13.94;
\]
see \cite{AndMor}.  Thus, with the rescaled Laplacian
$\overline{\mathcal L}:=\mathcal L/\lambda^*$, the spectrum used by the
polynomial recurrences lies in $[0,1]$ without requiring the largest
eigenvalue.

For $\tau\in[0,1]$, define
\[
 c(\tau):=\bigl(0.20+0.60\tau,\,0.50+0.20\sin(2\pi\tau)\bigr)
\]
and let the noise-free signal
$\ff(\tau)=(f_1(\tau),\ldots,f_N(\tau))^\top$ be given by the sum of three terms,
\begin{align}
 f_n(\tau)
 :=\;&0.55\bigl(1+0.15\cos(2\pi\tau)\bigr)
       (x_n^2+y_n^2-1) \notag\\
 &+0.35\exp\!\left(
       -\frac{\lVert z_n-c(\tau)\rVert_2^2}{2\cdot0.12^2}
       \right) \notag\\
 &+0.45\bigl(1+\cos(2\pi\tau)\bigr)
       {\bf 1}_{\{\lVert z_n-c(\tau)\rVert_2\leq0.18\}}.
 \label{eq:numerical-sharp-signal}
\end{align}

We use  $M=101$ equally spaced times $\tau_m=m/100$, $m=0,\ldots,100$,
and $B=100$ independent noise perturbations:
\begin{equation}
 \ff^{\eta,(b)}_m=\ff(\tau_m)+\boldsymbol\eta^{(b)}_m,
 \qquad
 \boldsymbol\eta^{(b)}_m\sim
 \mathcal N\!\left(0,\sigma(\tau_m)^2I_N\right),
 \qquad
 \sigma(\tau):=0.08+0.05\sin(2\pi\tau).
 \label{eq:numerical-noise-model}
\end{equation}
Throughout this section, the superscript $\eta$ indicates dependence on the
noise realization, both for the observed data and for the reconstructed signal.
The noise standard deviation is specified by $\sigma(\tau)$. For a single
observation and its reconstruction, we write $\ff^\eta$ and $\uu^\eta$,
respectively.
The noises are independent over $b$ and $m$, and the standard deviation
ranges from $0.03$ to $0.13$.  The graph, noise-free signal, and noise
realizations are held fixed across all comparisons.  Figure~\ref{fig:sensor-graph}
shows the graph and two representative instances of the noise-free signal.
For reconstructions
$\uu^\eta=\{\uu^{\eta,(b)}_m\}_{b,m}$, we report
\begin{equation}
 \operatorname{MSE}(\uu^\eta)
 :=\frac{1}{BMN}\sum_{b=1}^{B}\sum_{m=0}^{M-1}
 \bigl\lVert\uu^{\eta,(b)}_m-\ff(\tau_m)\bigr\rVert_2^2.
 \label{eq:numerical-mse}
\end{equation}
For each realization $b$, $\operatorname{MSE}_b(\uu^\eta)$ denotes the
average in \eqref{eq:numerical-mse} restricted to that realization; method
comparisons use paired differences of these errors.

\begin{figure}[t]
 \centering
 \includegraphics[width=\textwidth]{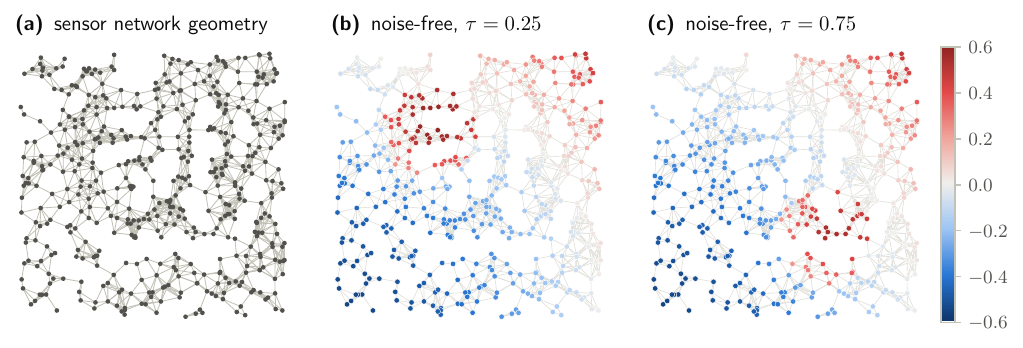}
 \caption{Test problem.  Panel~(a) shows the $500$-vertex sensor graph,
 with uniformly gray vertices indicating geometry only;
 panels~(b) and~(c) show the noise-free signal
 \eqref{eq:numerical-sharp-signal} at $\tau=0.25$ and $\tau=0.75$,
 respectively.    The color scale on the right applies
 to panels~(b) and~(c).}
 \label{fig:sensor-graph}
\end{figure}

\subsection{Kernels, uniform approximation, and stability}
\label{subsec:numerical-stability}

For $\gamma\in[5,800]$ and $\mu\in[0,1]$, consider the complementary
kernels
\begin{equation}
 h_\gamma(\mu):=\frac{1}{\sqrt{1+\gamma\mu^2}},
 \qquad
 g_\gamma(\mu):=\frac{\sqrt\gamma\,\mu}
                         {\sqrt{1+\gamma\mu^2}}.
 \label{eq:numerical-kernels}
\end{equation}
Set
\[
 H_\gamma:=h_\gamma(\overline{\mathcal L}),
 \qquad
 G_\gamma:=g_\gamma(\overline{\mathcal L}),
 \qquad
 W_\gamma:=\begin{pmatrix}H_\gamma\\G_\gamma\end{pmatrix}.
\]
This is precisely the overall graph wavelet transform \eqref{GWT} with
$r=1$ and $s_1=1$: we replace
$\mathcal L$ by the rescaled Laplacian $\overline{\mathcal L}$, so that the
spectral variable becomes $\mu=\lambda/\lambda^*\in[0,1]$, and regard
$\gamma$ as the parameter governing the kernels.  Since $h_\gamma(0)=1$
and $g_\gamma(0)=0$, the kernels satisfy the conditions imposed on the
scaling and wavelet kernels in Section~\ref{Back SGT}; thus $H_\gamma$ is
the graph scaling operator and $G_\gamma$ is the graph wavelet operator.  Once a
parameter path $\gamma=\gamma(\tau)$ is chosen, the notation of
\eqref{GWT} becomes $W[\tau]=W_{\gamma(\tau)}$.  In the experiment this path
is selected at the sampled times by SURE.
Because $h_\gamma^2+g_\gamma^2=1$, one has
$W_\gamma^*W_\gamma=I_N$ for every $\gamma$: the exact transform is an
isometry, and $W_\gamma^\dagger=W_\gamma^*$.

Both kernels in \eqref{eq:numerical-kernels} are $C^\infty$ on $\mathbb R\times(0,\infty)$. Consequently, for every fixed $0<\gamma_{\min}<\gamma_{\max}<\infty$, their restrictions to $[0,1]\times[\gamma_{\min},\gamma_{\max}]$ satisfy the variation-type condition after an affine change of variables to $[-1,1]^2$, by Remark~\ref{rem C1}. Theorem~\ref{arzela} therefore gives Chebyshev expansions in $\mu$ that converge uniformly with respect to $\gamma$ on this rectangle, while Theorem~\ref{K-error} provides an $O(K^{-m})$ truncation error for every fixed $m\geq1$, uniformly on the same rectangle. The error constant may depend on $m$, $\gamma_{\min}$, and $\gamma_{\max}$, but not on $K$ or on $\gamma$ within this interval. Moreover, the identity $h_\gamma^2+g_\gamma^2=1$ gives the lower frame bound $A=1$, so the small-error condition in Proposition~\ref{stability} holds for sufficiently large degrees. Uniformity in $\gamma$ allows these estimates to be evaluated at the SURE-selected parameter values in this interval without any regularity assumption on the selected sequence. The parameter interval $[5,800]$ used below is one such choice.

Let $p_{h,K}(\mu,\gamma)$ and $p_{g,K}(\mu,\gamma)$ be the degree-$K$
shifted Chebyshev approximants in $\mu$.  Their coefficients, which depend on
$\gamma$, are computed by an $8192$-point Gauss--Chebyshev rule.  The same
degree is used on the full rectangle $[0,1]\times[5,800]$.  Write
\[
 P_{h,K,\gamma}:=p_{h,K}(\overline{\mathcal L},\gamma),
 \qquad
 P_{g,K,\gamma}:=p_{g,K}(\overline{\mathcal L},\gamma),
 \qquad
 \widetilde W_{K,\gamma}:=
 \begin{pmatrix}P_{h,K,\gamma}\\P_{g,K,\gamma}\end{pmatrix}.
\]
Accordingly, $\widetilde W_{K,\gamma}$ is the corresponding instance of
the approximate transform $\widetilde W[\tau]$ introduced in
Section~\ref{Back Appl}. 

We evaluate the errors on an independent grid of $4001$ spectral points
and $401$ parameter values, augmenting the spectral grid by the
eigenvalues of $\overline{\mathcal L}$.  Define
\begin{align*}
 \widehat\varepsilon_{h,K}
 &:=\max_{\mu,\gamma}|p_{h,K}(\mu,\gamma)-h_\gamma(\mu)|,
 &
 \widehat\varepsilon_{g,K}
 &:=\max_{\mu,\gamma}|p_{g,K}(\mu,\gamma)-g_\gamma(\mu)|,\\
 \widehat\varepsilon_K
 &:=\bigl(\widehat\varepsilon_{h,K}^2
          +\widehat\varepsilon_{g,K}^2\bigr)^{1/2},
 &
 \widehat\delta_K
 &:=\max_\gamma\lVert \widetilde W_{K,\gamma}-W_\gamma\rVert_2,\\
 s_K&:=\min_\gamma\sigma_{\min}(\widetilde W_{K,\gamma}),
 &
 \widehat\rho_K
 &:=\max_\gamma
 \lVert \widetilde W_{K,\gamma}^\dagger W_\gamma-I_N\rVert_2.
\end{align*}
Since the exact lower frame bound is one and the validation grid includes
the spectrum, the argument of Proposition~\ref{stability} gives the following
inequalities on this grid, provided $\widehat\varepsilon_K<1$:
\begin{equation}
 \widehat\delta_K\leq\widehat\varepsilon_K,
 \qquad
 1-s_K\leq\widehat\varepsilon_K,
 \qquad
 \widehat\rho_K\leq
 \frac{\widehat\varepsilon_K}{1-\widehat\varepsilon_K}.
 \label{eq:numerical-stability-bounds}
\end{equation}

\begin{table}[ht]
\centering
\footnotesize
\setlength{\tabcolsep}{3.20pt}
\begin{tabular}{c|c|c|c|c}
\hline
$K$ & $\widehat\varepsilon_K$ & $\widehat\delta_K$
& $1-s_K$ [$\widehat\varepsilon_K$]
& $\widehat\rho_K$
  [$\widehat\varepsilon_K/(1-\widehat\varepsilon_K)$]\\
\hline
$8$
& $6.38{\times}10^{-2}$ & $5.48{\times}10^{-2}$
& $5.23{\times}10^{-2}$ [$6.38{\times}10^{-2}$]
& $5.51{\times}10^{-2}$ [$6.82{\times}10^{-2}$]\\
$16$
& $6.12{\times}10^{-3}$ & $5.67{\times}10^{-3}$
& $2.38{\times}10^{-3}$ [$6.12{\times}10^{-3}$]
& $2.72{\times}10^{-3}$ [$6.16{\times}10^{-3}$]\\
$24$
& $5.68{\times}10^{-4}$ & $5.27{\times}10^{-4}$
& $4.25{\times}10^{-4}$ [$5.68{\times}10^{-4}$]
& $4.25{\times}10^{-4}$ [$5.68{\times}10^{-4}$]\\
$32$
& $5.95{\times}10^{-5}$ & $5.17{\times}10^{-5}$
& $5.05{\times}10^{-5}$ [$5.95{\times}10^{-5}$]
& $5.06{\times}10^{-5}$ [$5.95{\times}10^{-5}$]\\
$40$
& $6.34{\times}10^{-6}$ & $5.85{\times}10^{-6}$
& $3.04{\times}10^{-6}$ [$6.34{\times}10^{-6}$]
& $3.04{\times}10^{-6}$ [$6.34{\times}10^{-6}$]\\
\hline
\end{tabular}
\caption{Sampled uniform kernel errors and stability quantities on
$[0,1]\times[5,800]$.  Brackets contain the bounds in
\eqref{eq:numerical-stability-bounds}.}
\label{tab:sure-stability}
\end{table}

\begin{figure}[t]
 \centering
 \includegraphics[width=0.92\textwidth]{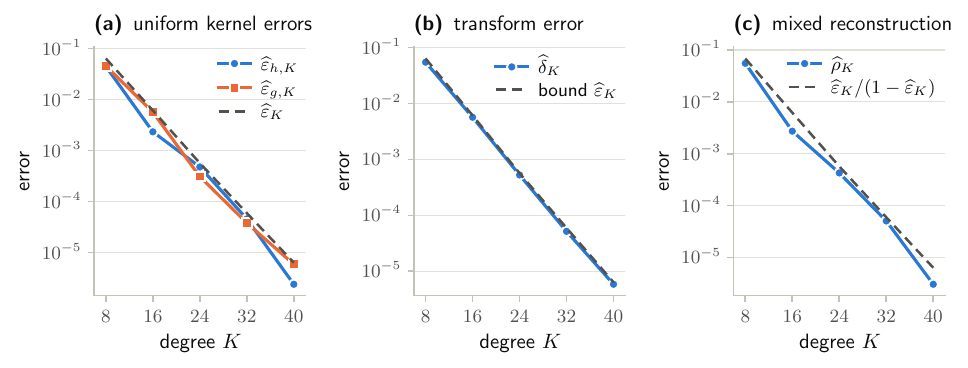}
 \caption{Sampled approximation and stability quantities on
 $[0,1]\times[5,800]$.  (a) The kernel errors
 $\widehat\varepsilon_{h,K}$ and $\widehat\varepsilon_{g,K}$ and their
 combined value $\widehat\varepsilon_K$.  (b) The transform error
 $\widehat\delta_K$ and its upper bound $\widehat\varepsilon_K$.  (c) The
 mixed reconstruction error $\widehat\rho_K$ and its upper bound
 $\widehat\varepsilon_K/(1-\widehat\varepsilon_K)$.}
 \label{fig:sensor-benchmark}
\end{figure}

All three inequalities in \eqref{eq:numerical-stability-bounds} hold at
every tested degree $K$.  In particular, $\widehat\varepsilon_K$ decreases from
$6.38\times10^{-2}$ at $K=8$ to $6.34\times10^{-6}$ at $K=40$, while the
minimum observed value of $s_K$ is $0.95$ to the reported precision.  These
results show that the reconstruction systems are well conditioned at all
tested parameter values. Positive definiteness is also checked separately
at every selected parameter used below. The three panels of
Figure~\ref{fig:sensor-benchmark} show the same approximation, transform,
and reconstruction quantities on logarithmic scales.

\subsection{Soft thresholding of the graph wavelet coefficients}
\label{subsec:sure-shrinkage}

We now denoise with both components of the exact transform.  Fix $\gamma$, a
threshold factor $\kappa\geq0$, and a noise level $\sigma$.  Under
\eqref{eq:numerical-noise-model}, the $n$th graph wavelet coefficient
$(G_\gamma\ff^\eta)_n$ has standard deviation
$\sigma\bigl((G_\gamma^2)_{nn}\bigr)^{1/2}$, so we compare each coefficient
with its own noise scale through the thresholds
\[
 \lambda_{\gamma,n}:=
 \kappa\sigma\bigl((G_\gamma^2)_{nn}\bigr)^{1/2},
 \qquad n=1,\ldots,N.
\]
 The soft-thresholding function
\cite{Donoho1995} acts componentwise on a vector $d\in\RR^N$ as
\begin{equation}
 \bigl(\mathcal S_{\lambda_\gamma}(d)\bigr)_n
 :=\operatorname{sign}(d_n)
 \bigl(|d_n|-\lambda_{\gamma,n}\bigr)_+.
 \label{eq:numerical-soft-threshold}
\end{equation}
We keep the graph scaling coefficients $H_\gamma\ff^\eta$ unchanged, threshold
the graph wavelet coefficients $G_\gamma\ff^\eta$, and reconstruct by solving
the associated normal equations.  Since $W_\gamma^*W_\gamma=I_N$, these
reduce to applying the adjoint, and the resulting reconstruction map is
\begin{equation}
 \mathcal D_{\gamma,\kappa,\sigma}(\ff^\eta)
 :=W_\gamma^*
 \begin{pmatrix}
 H_\gamma\ff^\eta\\
 \mathcal S_{\lambda_\gamma}(G_\gamma\ff^\eta)
 \end{pmatrix}
 =H_\gamma^2\ff^\eta
  +G_\gamma\mathcal S_{\lambda_\gamma}(G_\gamma\ff^\eta).
 \label{eq:numerical-exact-shrinkage}
\end{equation}
Note that at $\kappa=0$ the map
returns $\ff^\eta$, while suppressing all graph wavelet coefficients returns
the Tikhonov reconstruction
\[
 H_\gamma^2\ff^\eta
 =(I_N+\gamma\overline{\mathcal L}^{\,2})^{-1}\ff^\eta
 =\operatorname*{argmin}_{u\in\RR^N}
 \left\{\lVert u-\ff^\eta\rVert_2^2
       +\gamma\lVert\overline{\mathcal L}u\rVert_2^2\right\}.
\]

We select $\gamma$ and $\kappa$ from the noisy data alone  by SURE~\cite{Stein1981}.  For a fixed pair
$(\gamma,\kappa)$, the divergence of
\eqref{eq:numerical-exact-shrinkage} exists almost everywhere and equals
\begin{equation}
 \operatorname{div}\mathcal D_{\gamma,\kappa,\sigma}(\ff^\eta)
 =\operatorname{tr}(H_\gamma^2)
 +\sum_{n=1}^N
 {\bf 1}_{\{|(G_\gamma\ff^\eta)_n|>\lambda_{\gamma,n}\}}
 (G_\gamma^2)_{nn},
 \label{eq:numerical-shrinkage-divergence}
\end{equation}
and the associated risk estimate is
\begin{align}
 \operatorname{SURE}(\ff^\eta;\gamma,\kappa,\sigma)
 :=\;&\frac1N
 \bigl\lVert\mathcal D_{\gamma,\kappa,\sigma}(\ff^\eta)
       -\ff^\eta\bigr\rVert_2^2 +\frac{2\sigma^2}{N}
 \operatorname{div}\mathcal D_{\gamma,\kappa,\sigma}(\ff^\eta)
 -\sigma^2.
 \label{eq:numerical-sure}
\end{align}
For each fixed pair, \eqref{eq:numerical-sure} is an unbiased estimate of
the mean squared error per vertex under~\eqref{eq:numerical-noise-model}.
The noise-free signal $\ff$ never enters the parameter selection.

The candidate sets are
\begin{equation}
 \Gamma:=\left\{
 5\left(\frac{800}{5}\right)^{j/24}:j=0,\ldots,24
 \right\},
 \qquad
 \mathcal K:=\{0.50,0.75,\ldots,4.00\}.
 \label{eq:numerical-parameter-grids}
\end{equation}
Selection is repeated independently for every realization $b$.  The
\emph{time-varying method} (the one we propose) uses one threshold
factor over the complete time interval, while the parameter $\gamma$ may
change with time:
\begin{align}
 \widehat\kappa_b^{\rm tv}
 &\in\operatorname*{argmin}_{\kappa\in\mathcal K}
 \sum_{m=0}^{M-1}\min_{\gamma\in\Gamma}
 \operatorname{SURE}
 \bigl(\ff^{\eta,(b)}_m;\gamma,\kappa,\sigma(\tau_m)\bigr),
 \label{eq:numerical-adaptive-kappa}\\
 \widehat\gamma_{b,m}^{\rm tv}
 &\in\operatorname*{argmin}_{\gamma\in\Gamma}
 \operatorname{SURE}
 \bigl(\ff^{\eta,(b)}_m;\gamma,\widehat\kappa_b^{\rm tv},
       \sigma(\tau_m)\bigr).
 \label{eq:numerical-adaptive-gamma}
\end{align}
The comparison, the \emph{time-independent method}, uses one pair for all
sampled times:
\begin{equation}
 (\widehat\gamma_b^{\rm ti},\widehat\kappa_b^{\rm ti})
 \in\operatorname*{argmin}_{(\gamma,\kappa)\in\Gamma\times\mathcal K}
 \sum_{m=0}^{M-1}
 \operatorname{SURE}
 \bigl(\ff^{\eta,(b)}_m;\gamma,\kappa,\sigma(\tau_m)\bigr).
 \label{eq:numerical-fixed-selection}
\end{equation}
In the notation of \eqref{GWT}, the time-varying method applies
$W_{\widehat\gamma_{b,m}^{\rm tv}}$ at time $\tau_m$, whereas the
time-independent method applies $W_{\widehat\gamma_b^{\rm ti}}$ at every
sampled time.  The two methods use the same reconstruction map
\eqref{eq:numerical-exact-shrinkage}; their only difference is whether
$\gamma$ may vary with the sampled time.  Inserting the selected parameters
into \eqref{eq:numerical-exact-shrinkage} defines the reconstructions
evaluated below: for every realization $b$ and sampled time $\tau_m$,
\begin{equation}
 \uu^{\eta,(b)}_{{\rm tv},m}
 :=\mathcal D_{\widehat\gamma_{b,m}^{\rm tv},\,
               \widehat\kappa_b^{\rm tv},\,
               \sigma(\tau_m)}
   \bigl(\ff^{\eta,(b)}_m\bigr),
 \qquad
 \uu^{\eta,(b)}_{{\rm ti},m}
 :=\mathcal D_{\widehat\gamma_b^{\rm ti},\,
               \widehat\kappa_b^{\rm ti},\,
               \sigma(\tau_m)}
   \bigl(\ff^{\eta,(b)}_m\bigr),
 \label{eq:numerical-selected-reconstructions}
\end{equation}
and we write
$\uu^\eta_{\rm tv}:=\{\uu^{\eta,(b)}_{{\rm tv},m}\}_{b,m}$ and
$\uu^\eta_{\rm ti}:=\{\uu^{\eta,(b)}_{{\rm ti},m}\}_{b,m}$ for the
corresponding families, to which the mean squared error
\eqref{eq:numerical-mse} applies.

\subsection{Chebyshev denoising and convergence}
\label{subsec:numerical-polynomial-shrinkage}

We now replace the exact operators in both SURE-selected soft-thresholding
methods by their degree-$K$ Chebyshev approximants,
$K\in\{8,16,24,32,40\}$.  The parameter selection of
Subsection~\ref{subsec:sure-shrinkage} is held fixed across degrees, and the
exact reconstructions in \eqref{eq:numerical-selected-reconstructions}
serve as references.  Thus variation with $K$ isolates the approximation
of the transform and its pseudoinverse, while comparison of the two methods
at a common degree tests whether the benefit of time-adaptive parameter
selection persists.  At a given realization and time, we abbreviate the
operators of the time-varying method by
\[
 P_h:=p_{h,K}\bigl(\overline{\mathcal L},\widehat\gamma_{b,m}^{\rm tv}\bigr),
 \qquad
 P_g:=p_{g,K}\bigl(\overline{\mathcal L},\widehat\gamma_{b,m}^{\rm tv}\bigr),
 \qquad
 \widetilde W_K:=\begin{pmatrix}P_h\\P_g\end{pmatrix},
\]
so that $\widetilde W_K=\widetilde W_{K,\widehat\gamma_{b,m}^{\rm tv}}$ in
the notation of Subsection~\ref{subsec:numerical-stability}.  In words,
$P_h$ and $P_g$ are the shifted Chebyshev approximants of the kernels
$h_\gamma$ and $g_\gamma$ applied to the rescaled Laplacian
$\overline{\mathcal L}$, with the selected parameter in the second
variable: the degree-$K$ truncations of the uniform expansions with
time-varying coefficients provided by Theorem~\ref{arzela}, on which the
approximate transform of Section~\ref{Back Appl} is built.  Since
$\widehat\gamma_{b,m}^{\rm tv}$ changes with the sampled time, the
polynomials of the time-varying method change from one sampled time to
the next, while their degree $K$ is the same at all times.  The
time-independent method is treated in the same way, with
$(\widehat\gamma_b^{\rm ti},\widehat\kappa_b^{\rm ti})$ in place of the
time-varying parameters.  The approximate coefficient vector is
\begin{equation}
 z_K:=
 \begin{pmatrix}
 P_h\ff^{\eta,(b)}_m\\
 \mathcal S_{\lambda_K}(P_g\ff^{\eta,(b)}_m)
 \end{pmatrix},
 \qquad
 (\lambda_K)_n:=
 \widehat\kappa_b^{\rm tv}\sigma(\tau_m)
 \bigl((P_g^2)_{nn}\bigr)^{1/2},
 \label{eq:numerical-polynomial-coefficients}
\end{equation}
where the diagonal entries $(P_g^2)_{nn}$ require no spectral information:
since $P_g$ is symmetric, being a polynomial in $\overline{\mathcal L}$,
the entry $(P_g^2)_{nn}$ equals the squared Euclidean norm of the $n$-th
column of $P_g$, and each column is obtained by applying the same Chebyshev
recurrence to the corresponding canonical basis vector
$\boldsymbol{\del}_n$, that is, by $K$
multiplications by the sparse matrix $\overline{\mathcal L}$.  The factors
$\bigl((P_g^2)_{nn}\bigr)^{1/2}$ can therefore be precomputed once for each
degree and selected kernel parameter, without an eigendecomposition; the
thresholds are obtained by multiplying these factors by
$\widehat\kappa_b^{\rm tv}\sigma(\tau_m)$. In the experiment we evaluate
the factors from the eigendecomposition, consistently with
the reference calculations of Subsection~\ref{subsec:sure-shrinkage}; the
two evaluations agree up to rounding errors, so the test measures only the
approximation of the transform and of the reconstruction, with the
parameter selection fixed.

We reconstruct from $z_K$ by solving the associated normal equations:
\begin{equation}
 (P_h^2+P_g^2)\uu^{\eta,(b)}_{K,m}
 =P_h(z_K)_h+P_g(z_K)_g.
 \label{eq:numerical-canonical-normal-equation}
\end{equation}
The matrix $P_h^2+P_g^2=\widetilde W_K^*\widetilde W_K$ is checked to be
positive definite at every reported degree and selected parameter.
Consequently, the solution is the unique least-squares reconstruction
from $z_K$, equivalently
\begin{equation}
 \uu^{\eta,(b)}_{K,m}=\widetilde W_K^\dagger z_K.
 \label{eq:numerical-polynomial-reconstruction}
\end{equation}
As in \eqref{eq:numerical-exact-shrinkage}, both the retained scaling
coefficients and the thresholded wavelet coefficients enter the
reconstruction.  Here the kernels and coefficient noise scales are
replaced by their polynomial counterparts, and
$\widetilde W_K^*\widetilde W_K$ is generally no longer the identity.
Proposition~\ref{stability} controls the stability of this reconstruction
when the uniform kernel errors are sufficiently small.

We solve \eqref{eq:numerical-canonical-normal-equation} by conjugate
gradients, starting from zero and using a relative residual tolerance of
$10^{-10}$.  In the reported calculations, the polynomial operators are
evaluated using the eigendecomposition of $\overline{\mathcal L}$, and CG
is performed in that orthogonal eigenbasis, where the system matrix is
diagonal.  This change of basis preserves the CG iterates and residual
norms in exact arithmetic.  A sparse implementation evaluates the matrix
action through the degree-$2K$ Chebyshev expansion of $P_h^2+P_g^2$.

Let
$\uu^\eta_K:=\{\uu^{\eta,(b)}_{K,m}\}_{b,m}$ collect the
time-varying Chebyshev reconstructions
\eqref{eq:numerical-polynomial-reconstruction}, and recall from
\eqref{eq:numerical-selected-reconstructions} the exact time-varying
family $\uu^\eta_{\rm tv}$, built from the same selected parameters.
We measure the distance between the approximate and the exact
reconstructions by
\begin{equation}
 \mathcal E_K
 :=\left(
 \frac{\sum_{b,m}\lVert
 \uu^{\eta,(b)}_{K,m}-
 \uu^{\eta,(b)}_{{\rm tv},m}\rVert_2^2}
 {\sum_{b,m}\lVert\uu^{\eta,(b)}_{{\rm tv},m}\rVert_2^2}
 \right)^{1/2}
 \label{eq:numerical-shrinkage-relative-error}
\end{equation}
and the relative difference of the mean squared errors,
\[
 \Delta_K^{\rm MSE}:=
 \frac{\bigl|
 \operatorname{MSE}(\uu^\eta_K)
 -\operatorname{MSE}(\uu^\eta_{\rm tv})
 \bigr|}
 {\operatorname{MSE}(\uu^\eta_{\rm tv})}.
\]

\begin{table}[ht]
\centering
\footnotesize
\setlength{\tabcolsep}{3.50pt}
\begin{tabular}{c|c|c|c|c|c}
\hline
$K$
& \shortstack{time-varying MSE\\($\times10^{-3}$)}
& \shortstack{time-independent MSE\\($\times10^{-3}$)}
& $\mathcal E_K$
& $\Delta_K^{\rm MSE}$
& \shortstack{CG iterations\\median [maximum]}\\
\hline
$8$
& $2.2887$ & $2.3912$ & $2.11{\times}10^{-3}$
& $3.39{\times}10^{-4}$ & $4$ [$8$]\\
$16$
& $2.2894$ & $2.3905$ & $2.37{\times}10^{-4}$
& $1.52{\times}10^{-5}$ & $3$ [$4$]\\
$24$
& $2.2895$ & $2.3905$ & $1.71{\times}10^{-5}$
& $5.87{\times}10^{-7}$ & $2$ [$3$]\\
$32$
& $2.2895$ & $2.3905$ & $1.01{\times}10^{-6}$
& $2.07{\times}10^{-7}$ & $2$ [$3$]\\
$40$
& $2.2895$ & $2.3905$ & $2.14{\times}10^{-7}$
& $4.59{\times}10^{-8}$ & $2$ [$2$]\\
\hline
\end{tabular}
\caption{Chebyshev reconstructions of degree $K$, computed with the
parameters selected once in Subsection~\ref{subsec:sure-shrinkage} and held
fixed across degrees, so only the approximation of the transform changes
with $K$.  The second and third columns show that the time-varying method
has lower mean squared error at every degree.
The fourth and fifth columns measure the convergence of
the time-varying Chebyshev reconstruction to its exact counterpart.  Iteration counts of the
conjugate gradient method (relative residual tolerance $10^{-10}$) are
summarized by their median and maximum for the time-varying method.}
\label{tab:sure-polynomial-convergence}
\end{table}

Table~\ref{tab:sure-polynomial-convergence} shows three things.  First, the
time-varying mean squared error lies below the time-independent one at
every degree,  already at $K=8$.  For the exact reference transform,
allowing time variation reduces the overall mean squared error by $4.23\%$,
with
$\operatorname{MSE}_b(\uu^\eta_{\rm ti})-
\operatorname{MSE}_b(\uu^\eta_{\rm tv})>0$ for all $B=100$
experiments.
Second, $\mathcal E_K$ decreases from $2.11\times10^{-3}$ at
$K=8$ to $2.14\times10^{-7}$ at $K=40$, and $\Delta_K^{\rm MSE}$ from
$3.39\times10^{-4}$ to $4.59\times10^{-8}$: as the sampled uniform kernel
errors decrease, the complete nonlinear reconstruction, thresholding
included, approaches its exact counterpart. Third, the conjugate gradient method
never needs more than $8$ iterations at $K=8$, nor more than $2$ at
$K=40$: the spectra of the operators $P_h^2+P_g^2$ already lie in
$[0.89,1.10]$ at $K=8$ and deviate from one by at most
$6.09\times10^{-6}$ at $K=40$, in agreement with
Proposition~\ref{stability} and Table~\ref{tab:sure-stability}. 

\section{Conclusions}\label{sec:conclusions}

We have proved a uniform Chebyshev approximation theorem on the square for continuous functions satisfying a variation-type condition (Theorem~\ref{arzela}), together with a double-series version (Theorem~\ref{unif in 2} and its corollary) and uniform error bounds. On this basis, the Chebyshev reconstruction scheme of \cite{Hammond} extends to time-varying signals on graphs: the scaling and wavelet kernels are approximated by expansions with time-varying coefficients whose degrees do not depend on time, the composition $\tilde{W}^*[\tau]\tilde{W}[\tau]$ satisfies the same explicit coefficient formulas as in the time-independent case, and the  reconstruction is stable, with explicit bounds in terms of the uniform kernel errors (Proposition~\ref{stability}).

The experiments of Section~\ref{sec:numerical} confirm this picture on a sensor network, for kernels that are analytic in the pair formed by the spectral variable and the regularization parameter: the sampled uniform errors decay rapidly with the degree, and the corresponding stability inequalities from Proposition~\ref{stability} hold on the validation grid at all tested degrees. The complete transform is then used for denoising, by soft thresholding the graph wavelet coefficients and reconstructing by solving the associated normal equations. The resulting time-varying method consistently attains a smaller mean squared error than the time-independent method.

\appendix

\section{Proof of Lemma~\ref{form for 2}}\label{AppA}
Recall from \eqref{che} that
\[T_{m,n}(x,y)=u^mv^{-n}+u^mw^{-n}+v^mu^{-n}+v^mw^{-n}+w^mu^{-n}+w^mv^{-n}.\]
Straightforward calculations yield
\begin{align*}
&T_{m,n}(x,y)T_{k,\ell}(x,y)\\
&=u^{m+k}v^{-n-\ell}+u^{m+k}v^{-n}w^{-\ell}+u^{m-\ell}v^{-n+k}+u^{m}v^{k-n}w^{-\ell}+u^{m-\ell}v^{-n}w^{k}+u^{m}v^{-n-\ell}w^{k}\\
&+u^{m+k}v^{-\ell}w^{-n}+u^{m+k}w^{-n-\ell}+u^{m-\ell}v^{k}w^{-n}+u^{m}v^{k}w^{-n-\ell}+u^{m-\ell}w^{k-n}+u^{m}v^{-\ell}w^{k-n}\\
&+u^{k-n}v^{m-\ell}+u^{k-n}v^{m}w^{-\ell}+u^{-n-\ell}v^{m+k}+u^{-n}v^{m+k}w^{-\ell}+u^{-n-\ell}v^{m}w^{k}+u^{-n}v^{m-\ell}w^{k}\\
&+u^{k}v^{m-\ell}w^{-n}+u^{k}v^{m}w^{-n-\ell}+u^{-\ell}v^{m+k}w^{-n}+v^{m+k}w^{-n-\ell}+u^{-\ell}v^{m}w^{-n+k}+v^{m-\ell}w^{k-n}\\
&+u^{k-n}v^{-\ell}w^{m}+u^{k-n}w^{m-\ell}+u^{-n-\ell}v^{k}w^{m}+u^{-n}v^{k}w^{m-\ell}+u^{-n-\ell}w^{m+k}+u^{-n}v^{-\ell}w^{m+k}\\
&+u^{k}v^{-n-\ell}w^{m}+u^{k}v^{-n}w^{m-\ell}+u^{-\ell}v^{k-n}w^{m}+v^{-n+k}w^{m-\ell}+u^{-\ell}v^{-n}w^{m+k}+v^{-n-\ell}w^{m+k}.
\end{align*}
Note that
\begin{align*}
T_{m+k,n+\ell}(x,y)=& u^{m+k}v^{-n-\ell}+ u^{m+k}w^{-n-\ell}+v^{m+k}u^{-n-\ell}\\
&+v^{m+k}w^{-n-\ell}+w^{m+k}u^{-n-\ell}+w^{m+k}v^{-n-\ell},
\end{align*}
and
\begin{align*}
T_{m-\ell,n-k}(x,y)=&u^{m-\ell}v^{k-n}+u^{m-\ell}w^{k-n}+v^{m-\ell}u^{k-n}\\
&+v^{m-\ell}w^{k-n}+w^{m-\ell}u^{k-n}+w^{m-\ell}v^{k-n}.
\end{align*}
Furthermore, applying the property $uvw=1$, we obtain
\begin{align*}
T_{m+k+\ell,n-\ell}(x,y)&=u^{m+k+\ell}v^{\ell-n}+u^{m+k+\ell}w^{\ell-n}+v^{m+k+\ell}u^{\ell-n}\\
&+v^{m+k+\ell}w^{\ell-n}+w^{m+k+\ell}u^{\ell-n}+w^{m+k+\ell}v^{\ell-n}\\
&=u^{m+k}v^{-n}(uv)^{\ell}+u^{m+k}w^{-n}(uw)^\ell+v^{m+k}u^{-n}(vu)^\ell\\
&+v^{m+k}w^{-n}(vw)^\ell+w^{m+k}u^{-n}(wu)^\ell+w^{m+k}v^{-n}(wv)^\ell\\
&=u^{m+k}v^{-n}w^{-\ell}+u^{m+k}v^{-\ell}w^{-n}+u^{-n}v^{m+k}w^{-\ell}\\
&+u^{-\ell}v^{m+k}w^{-n}+u^{-n}v^{-\ell}w^{m+k}+u^{-\ell}v^{-n}w^{m+k}.
\end{align*}
Similarly,
\begin{align*}
T_{k-n+\ell,-m-\ell}(x,y)=&u^{m}v^{k-n}w^{-\ell}+u^{m}v^{-\ell}w^{k-n}+u^{k-n}v^{m}w^{-\ell}\\
&+u^{-\ell}v^{m}w^{-n+k}+u^{k-n}v^{-\ell}w^{m}+u^{-\ell}v^{k-n}w^{m},
\end{align*}
\begin{align*}
T_{m-\ell+n,-k-n}(x,y)=&u^{m-\ell}v^{-n}w^{k}+u^{m-\ell}v^{k}w^{-n}+u^{-n}v^{m-\ell}w^{k}\\
&+u^{k}v^{m-\ell}w^{-n}+u^{-n}v^{k}w^{m-\ell}+u^{k}v^{-n}w^{m-\ell},
\end{align*}
\begin{align*}
T_{m+n+\ell,-k-n-\ell}(x,y)=&u^{m}v^{-n-\ell}w^{k}+u^{m}v^{k}w^{-n-\ell}+u^{-n-\ell}v^{m}w^{k}\\
&+u^{k}v^{m}w^{-n-\ell}+u^{-n-\ell}v^{k}w^{m}+u^{k}v^{-n-\ell}w^{m}.
\end{align*}
Summing the six expansions above and comparing with the expansion of $T_{m,n}(x,y)T_{k,\ell}(x,y)$ yields the assertion.

\section*{Acknowledgements}
This work was supported by the Italian MUR project PRIN 2022 (Progetto di ricerca di rilevante interesse nazionale) 20227TRY8H ``TIme-varying signals on Graphs: REal and COmplex methods'' (TIGRECO). D.~B.~is supported by the Startup Fund of Sun Yat-sen University.

\bibliographystyle{plainurl}
\bibliography{main_LATEST}

\end{document}